\documentclass{article}
\usepackage[a4paper, left=3cm, right=3cm, top=3cm]{geometry}
\usepackage{amssymb}
\usepackage{amsmath}
\usepackage{amsthm}
\usepackage{graphicx}
\usepackage{subcaption}
\usepackage{bbm} 

\title{On the approaching time towards the attractor of differential equations 
perturbed by small noise}
\author{
	Isabell Vorkastner\thanks{
	Institut f\" ur Mathematik, MA 7-5, Fakult\" at II, Technische Universit\" at Berlin, 
	Stra\ss e des 17. Juni 136, \newline 10623 Berlin, Germany; 
	e-mail: \texttt{vorkastn@math.tu-berlin.de}  } 
	 }
\date{\today}

\begin{document}

\maketitle

\theoremstyle{plain}
\newtheorem{theorem}{Theorem}[section]
\newtheorem{lemma}[theorem]{Lemma}
\newtheorem{proposition}[theorem]{Proposition}
\newtheorem{corollary}[theorem]{Corollary}
\theoremstyle{definition}
\newtheorem{definition}[theorem]{Definition}
\newtheorem{example}[theorem]{Example}
\newtheorem{remark}[theorem]{Remark}

\begin{abstract}
	\noindent
	We estimate the time a point or set, respectively, requires to approach the attractor
	of a radially symmetric gradient type stochastic differential equation driven by small
	noise.
	Here, both of these times tend to infinity as the noise gets small. 
	However, the rates at which they go to infinity differ significantly.
	In the case of a set approaching the attractor, 
	we use large deviation techniques to show that this time increases exponentially.
	In the case of a point approaching the attractor, we apply a time change and compare 
	the accelerated process to another process and obtain that this time increases merely linearly.
\end{abstract}

\section{Introduction}

Noise induced stabilization is an interesting phenomena to occur.
Here, the long-time dynamics in the absence of noise are not asymptotically stable while
the addition of noise stabilizes the dynamics.
One way to quantify the asymptotic behavior is to analyze (random) attractors.
The most common kinds of (random) attractors are (random) point and set attractors.
While (random) point attractors need to attract any single point, 
(random) set attractors even need to attract compact sets uniformly.
If an (random) attractor is a single (random) point, the long-time dynamics are asymptotically globally stable.
\\
In the case of noise induced stabilization, we can anticipate that the time required for a point or set, respectively, to approach the
attractor goes to infinity as the noise, which stabilizes the system, gets small.
Our aim is to estimates these times and provide the rates at which they tend to infinity.
\\
We consider a radially symmetric gradient type stochastic differential equation, i.e. 
\begin{align}
	\label{original_SDE}
	d X_t^\varepsilon = - \nabla U ( X_t^\varepsilon )  dt 
	+ \sqrt{\varepsilon} dW_t  \quad \textrm{on } \mathbb{R}^d
\end{align}
where $d \geq 2$, $\varepsilon>0$, $W_t$ is a $d$-dimensional Brownian motion and $U(x) = u(|x|^2)$ for all $x \in \mathbb{R}^d$ 
and some twice differentiable convex function $u: [0, \infty) \rightarrow \mathbb{R}$ 
attaining its unique minimum in $(0,\infty)$. \\
In the absence of noise, the solution of the differential equation
\begin{align}
	\label{ODE}
	d x_t = - \nabla U ( x_t )  dt  \qquad \textrm{on } \mathbb{R}^d.
\end{align}
has a stable sphere, 
meaning that any point on this sphere is a fixed point and any point except $0$ converges towards the 
sphere under the dynamics of \eqref{ODE}. The point $0$ is also a fixed point.
In terms of attractors this means that the point attractor is the union
of $0$ and the stable sphere while the set attractor is the closed ball
of the same radius as the stable sphere centered at $0$.\\
An interesting phenomena occurs if one adds noise as in \eqref{original_SDE}. In \cite{Gess2016}
it was shown that under some general conditions on $U$, the attractor of \eqref{original_SDE} collapses
to a single random point. 
Therefore, the addition of noise stabilizes the system.
This phenomenon is called synchronization by noise.  \\
The question that arises is how fast this phenomena occurs for small noise.
By the above observations it is obvious that the time until a point or a set, respectively, approaches the 
attractor should go to infinity as the noise gets small. 
We estimate the rates at which these times tend to infinity and even show a significant difference between
the time a point or a set, respectively, requires to approach the attractor. 
This difference is due to the fact that a point can approach the attractor 
moving to the stable sphere and then 
along the sphere while a set can just approach the attractor 
if a point of the stable sphere moves close to zero.
\\
In section \ref{section_strong_syn}, we show that the time until a set approaches the attractor increases 
exponentially in $\varepsilon^{-1}$ using large deviation techniques similar to \cite[Section 5.7]{LDP1998}. 
We obtain a lower bound by taking the difference of the potential describing the costs of a point on the stable sphere to approach zero. 
Assuming that the differential equation \eqref{ODE} pushes all mass into a bounded set in finite time,
we get a upper bound for the time. This estimate in particular demonstrates the sharpness of our lower bound.\\
In section \ref{section_weak_syn}, we prove that the time until a point approaches the attractor is of 
order $\varepsilon^{-1}$ for dimension $d=2$. 
Here, we accelerate the process and compare the accelerated process to a process on the sphere
that is known to synchronize weakly.

\section{Preliminaries}

We consider the stochastic differential equation (SDE) \eqref{original_SDE} and assume that $-\nabla U$ satisfies a one-sided-Lipschitz condition, i.e.
there exists some $C>0$ such that
\begin{align*}
	\langle x- y , - \nabla U(x) + \nabla U(y) \rangle \leq C |x-y|^2
\end{align*}
for all $x,y \in \mathbb{R}^d$. Then, the SDE \eqref{original_SDE} has a unique solution. 
We denote by $X^\varepsilon :  [0, \infty) \times \Omega \times \mathbb{R}^d \rightarrow \mathbb{R}^d$ 
the solution of \eqref{original_SDE}. \\
We say that the SDE \eqref{original_SDE} \emph{is strongly contracting} 
if there exist $r,t >0$ such that
$|x_t(y)|\leq r$ for all $y \in \mathbb{R}^d$ where $x_t(y)$ is the solution of 
the deterministic differential equation started in $y$. 
Therefore, the SDE \eqref{original_SDE} is strongly contracting if and only if 
$\int_R^\infty | \nabla U(x) |^{-1} dx$ exists for some $R>0$. \\
Let $R^\ast \in (0, \infty)$ be the point where $u$ attains its minimum, i.e.
$ u(R^\ast) < u(x)$ for any $x \not= R^\ast$. 
We restrict the proofs in the following sections to the case $R^\ast = 1$. However, all results are
extendable to general $R^\ast \in (0, \infty)$ since $X_t^{\varepsilon} / R^\ast$ is of the postulated form.
\\
Attractors and synchronization are defined for random dynamical systems.
We restrict our definitions to a random dynamical system on $\mathbb{R}^d$, 
see \cite{Arnold1988} for a more general setting.
\begin{definition} [Metric Dynamical System]
	Let $(\Omega, \mathcal{F}, \mathbb{P})$ be a probability space 
	and $\theta = (\theta_t )_{t \in \mathbb{R}}$ be a group of 
	maps $\theta_t : \Omega \rightarrow \Omega$ satisfying
	\begin{enumerate}
	\item [(i)] $(\omega,t) \mapsto \theta_t (\omega)$ is 
		$(\mathcal{F} \otimes \mathcal{B}(\mathbb{R}), \mathcal{F})$-measurable,
	\item [(ii)] $\theta_0 (\omega) = \omega$ for all  $\omega \in \Omega$,
	\item [(iii)] $\theta_{s+t} = \theta_s \circ \theta_t$ for all $s,t \in \mathbb{R}$,
	\item [(iv)] $\theta_t$ has ergodic invariant measure $\mathbb{P}$.
	\end{enumerate}
	The collection $(\Omega, \mathcal{F}, \mathbb{P}, \theta)$ is then called 
	a \textit{metric dynamical system}.
\end{definition}

\begin{definition} [Random Dynamical System]
	Let $(\Omega, \mathcal{F}, \mathbb{P}, \theta)$ be a metric dynamical system. 
	Further, let $\varphi : \mathbb{R}_+ \times \Omega \times \mathbb{R}^d \rightarrow \mathbb{R}^d$ 
	be such that
	\begin{enumerate}
	\item [(i)] $\varphi$ is $(\mathcal{B}(\mathbb{R}_+) \otimes \mathcal{F} \otimes 
		\mathcal{B}(\mathbb{R}^d), \mathcal{B}(\mathbb{R}^d))$-measurable,
	\item [(ii)] $\varphi_0 (\omega, x) = x$ for all $x \in \mathbb{R}^d$, $\omega \in \Omega$,
	\item [(iii)] $\varphi_{t+s} (\omega,x)= \varphi_t ( \theta_s \omega , \varphi_s (\omega,x))$
		for all $x \in \mathbb{R}^d$, $t,s \geq 0$, $\omega \in \Omega$,
	\item [(iv)] $x \mapsto \varphi_s(\omega, x)$ is continuous for each $s \geq 0$ and $\omega \in \Omega$.
	\end{enumerate}
	The collection $(\Omega, \mathcal{F}, \mathbb{P}, \theta, \varphi )$ 
	is then called a \textit{random dynamical system} (RDS).
\end{definition}

\begin{definition}
	A family $\left\{ D (\omega) \right\}_{\omega \in \Omega}$ of non-empty subsets 
	of $\mathbb{R}^d$ is said to be
	\begin{enumerate}
		\item [(i)]  a \textit{random compact set} if it is $\mathbb{P}$-almost surely compact 
		and $\omega \mapsto \sup_{y \in D(\omega)} |x - y|$ 
		is $\mathcal{F}$-measurable for each $x \in \mathbb{R}^d$.
		\item [(ii)] \textit{$\varphi$-invariant} if for all $t \geq 0$
		\begin{align*}
			\varphi_t (\omega, D(\omega)) = D(\theta_t \omega)
		\end{align*}
		for almost all $\omega \in \Omega$.
	\end{enumerate}	 
\end{definition}

\begin{definition}[Attractor]
	Let $(\Omega, \mathcal{F}, \mathbb{P}, \theta, \varphi)$ be an RDS and 
	$A$ be $\varphi$-invariant random compact set $A$.
	\begin{enumerate}
	\item [(i)] $A$ is called a \textit{weak point attractor} if for every $x \in \mathbb{R}^d $
		\begin{align*}
			\lim_{t \rightarrow \infty} \inf_{a \in A(\omega)}
				\left| \varphi_t (\theta_{-t}\omega ,x) - a \right| =0 
				\qquad \textrm{in probability}.
		\end{align*}
	\item [(ii)] $A$ is called a \textit{weak attractor} if for every compact set $B \subset \mathbb{R}^d$ 
		\begin{align*}
			\lim_{t \rightarrow \infty} \sup_{x \in B} \inf_{a \in A(\omega)}
				\left| \varphi_t (\theta_{-t}\omega ,x) - a \right| =0 
				\qquad \textrm{in probability}.
		\end{align*}
	\end{enumerate}
\end{definition}

By \cite{Dimitroff2011}, 
the SDE \eqref{original_SDE} generates an RDS $(\Omega, \mathcal{F}, \mathbb{P}, \theta, X^\varepsilon)$
with respect to the canonical setup and this RDS has a weak attractor
. 
Note that every weak attractor is a weak point attractor. The converse is not true. \\
Here, the space $\Omega$ is $\mathcal{C}(\mathbb{R}, \mathbb{R}^d)$, $\mathcal{F}$ 
is the Borel $\sigma$-field,
$\mathbb{P}$ is the two-sided Wiener measure, $\mathcal{F}$ is the $\sigma$-algebra generated
by $W_u - W_v$ for $v \leq u$ , where $W_s: \Omega \rightarrow \mathbb{R}^d$ is defined
as $W_s(\omega) = \omega (s)$, and $\theta_t$ is the shift $(\theta_t \omega)(s) = \omega(s+t) - \omega(t)$.
Further, define $\mathcal{F}^+$ as the $\sigma$-algebra generated by $W_u - W_v$ for $0 \leq v \leq u$
and $\mathcal{F}^-$ as the $\sigma$-algebra generated by $W_u - W_v$ for $v \leq u \leq 0$.

\begin{definition}[Synchronization]
	\textit{Synchronization} occurs if there is a weak attractor $A(\omega)$ being 
	a singleton for $\mathbb{P}$-almost every $\omega \in \Omega$.
	\textit{Weak synchronization} is said to occur if there is a weak point attractor $A(\omega)$ being 
	a singleton for $\mathbb{P}$-almost every $\omega \in \Omega$.
\end{definition}

We do not require the RDS to synchronize (weakly) in order to get lower and upper bounds 
on the time required to approach the attractor. 
However, we differ between the smallest and largest distance to the attractor. 
Both quantities coincide if the RDS synchronize (weakly).
\\
The paper \cite{Gess2016} provides general conditions for the RDS associated 
to \eqref{original_SDE} to synchronize (weakly).
If the SDE \eqref{original_SDE} additionally satisfies $ u \in \mathcal{C}_{loc}^{3}$, 
$ \log ^+ |x| \exp ( -2 u(|x|^2) / \varepsilon ) \in L^1 (\mathbb{R}^d)$ and 
$|u''' (x)| \leq C (|x|^m +1)$ for some $m \in \mathbb{N}$, $C \geq 0$ and where $u'''$ is the third derivative of $u$, then the associated RDS synchronizes by \cite{Gess2016}.
The assumption $ \log ^+ |x| \exp ( -2 u(|x|^2) / \varepsilon ) \in L^1 (\mathbb{R}^d$
is in particular satisfied for a strongly contracting SDE \eqref{original_SDE}.
\\
Obviously, synchronization implies weak synchronization. It is left as an open problem in \cite{Gess2016}
whether any RDS associated to SDE \eqref{original_SDE} satisfying
$ \exp ( -2 u(|x|^2) / \varepsilon ) \in L^1 (\mathbb{R}^d)$ synchronize weakly.
\\
Denote by 
\begin{align*}
	B_r := \left\{ x \in \mathbb{R}^d : \left| x \right| <r \right\}
\end{align*}
the open ball of radius $r>0$ centered at $0$ and by
\begin{align*}
	S_r := \left\{ x \in \mathbb{R}^d : \left| x \right| =r \right\}
\end{align*}
the sphere of radius $r >0$ centered at $0$. For a set $M \subset \mathbbm{R}^d$ denote by $\bar{M}$ the
closure of the set $M$.

\section{Time required for a set to approach the attractor}
\label{section_strong_syn}

\subsection{Large deviation principle}
\label{section_LDP}
We use the large deviation principle (LDP) to describe the behavior of $X_t^\varepsilon$ for small 
$\varepsilon>0$. 
We aim to give an estimate on the time a set needs to approach the weak attractor. Observe that by 
\cite[Theorem 3.1]{Dimitroff2011}
there exists a weak attractor of the RDS associated to \eqref{original_SDE} 
and that the weak attractor is $\mathbb{P}$-almost surely unique by \cite[Lemma 1.3]{Gess2016}.
We denote by $A^{X,\varepsilon}$ the weak attractor.
\\
Let $\mu^\varepsilon_T$ be the probability measure induced by $\sqrt{\varepsilon} W_t$  on $C_0([0,T])$,
the space of all continuous functions $\phi:[0,T] \rightarrow \mathbb{R}^d$ such that $\phi(0)=0$
equipped with the supremum norm topology. By Schilder's theorem $\mu_T^\varepsilon$ satisfies an LDP with good rate function
\begin{align*}
	\hat{I}_T (g) = 
	\begin{cases}
		\frac{1}{2} \int_0^T | \dot{g} (t) |^2 dt, 
		& g \in \left\{ \int_0^t f(s) ds  : f \in L^2([0,T]) \right\}
		\\
		\infty, & \textrm{otherwise}
	\end{cases}.
\end{align*}
for $ g \in C_0 ([0,T])$.
The deterministic map $F_T: C_0 ([0,T]) \rightarrow C([0,T] \times \mathbb{R}^d , \mathbb{R}^d)$ is defined
by $f= F_T (g)$, where $ f$ is the semi-flow associated to
\begin{align}
	\label{semi-flow_with_g}
	f(t) = f(0) + \int_0^t - \nabla U(f(s)) ds +g(t), \quad t \in [0,T].
\end{align}
The LDP associated to the semi-flow $X_t^{\varepsilon}$ is therefore a direct application of the contraction 
principle with respect to the continuous map $F_T$. Therefore, $X_t^{\varepsilon}$ satisfies the LDP in 
$C([0,T] \times \mathbb{R}^d , \mathbb{R}^d)$ with good rate function
\begin{align*}
	I_T (\phi ) = \inf \left\{ \hat{I}_T(g): g \in C_0([0,T]) \textrm{ and } \phi = F_T(g) \right\}.
\end{align*}
for $\phi \in C([0,T] \times \mathbb{R}^d , \mathbb{R}^d )$. 
Define the stopping times
\begin{align*}
	\tau_{1,\delta}^\varepsilon &:= \inf \left\{ t \geq 0 : 
		\left| X_t^{\varepsilon}(x) -X_t^{\varepsilon}(y) \right| \leq \delta 
		\textrm{ for all } x,y \in S_1 \right\}, \\
	\tau_{2,\delta,M}^\varepsilon &:= \inf \left\{ t \geq 0 : 
		 \sup_{a \in A^{X, \varepsilon} (\theta_t \cdot) } | X_t^\varepsilon (x) - a|
		\leq \delta \textrm{ for all } x \in M \right\} \\
	\tau_{3,\delta}^\varepsilon &:= \inf \left\{ t \geq 0 : 
		\left| X_t^{\varepsilon}(x) -X_t^{\varepsilon}(y) \right| \leq \delta
		\textrm{ for all } x,y \in \mathbb{R}^d \right\}
\end{align*}
for $\delta>0$ and a set $M \subset \mathbb{R}^d$. 
Here, $\tau_{2,\delta,M}^\varepsilon$ describes the time the set $M$ needs to approach 
the attractor $A^{X,\varepsilon}$. 
Observe that 
$\tau_{1,2\delta}^\varepsilon \leq \tau_{2,\delta,M}^\varepsilon \leq \tau_{3,\delta}^\varepsilon$
for any $\delta>0$ and $S_1 \subset M \subset \mathbb{R}^d$.
\\
In the next subsection we use the LDP to show a lower bound 
for $\tau^\varepsilon_{1,\delta}$ and 
an upper bound for $\tau^\varepsilon_{3,\delta}$. 
We then conclude this section combining these estimates and showing that 
$\tau_{1,\delta}^\varepsilon$, $\tau_{2,\delta,M}^\varepsilon$ and $\tau_{3,\delta}^\varepsilon$ are 
roughly of order $\exp( V / \varepsilon )$ for some $V>0$.

\subsection{Lower bound for $\tau^\varepsilon_{1,\delta}$}
\label{section_lower_bound}

In this subsection we show a lower bound for $\tau^\varepsilon_{1,\delta}$. 
Using the gradient type form of the SDE \eqref{original_SDE}, we provide 
an upper bound for the probability that this stopping time is smaller than some deterministic time.
Afterwards, we use a similar approach as in \cite[Section 5.7]{LDP1998} to deduce that 
$\tau^\varepsilon_{1,\delta}$ is roughly greater than
$\exp(V/ \varepsilon)$ where $V>0$ is determined by the potential $U$.\\
Define the annulus
\begin{align*}
	D_{r,R} &:= \left\{ x \in \mathbb{R}^d : r < \left| x \right| < R  \right\}
\end{align*}
for $0 \leq r<R \leq \infty$. Moreover, denote by 
\begin{align*}
	\tau^\varepsilon (M,D) := \inf \left\{ t\geq 0: X_t^{\varepsilon}(x) \not\in D
	\textrm{ for some } x \in M \right\}
\end{align*}
the time until the semi-flow started in $M\subset \mathbb{R}^d$ leaves $D\subset \mathbb{R}^d$.
For $0 \leq r_1 < r_2 < r_3 \leq \infty$ with $r_1<1<r_3$ set
\begin{align*}
	V(r_1,r_2,r_3):= 
	2 \min \left\{ u(r_1^2) - u(\min \left\{ r_2^2,1 \right\}), 	u(r_3^2) - u(\max \left\{ r_2^2,1 \right\}) \right\} 
\end{align*}
where $u(\infty):= \lim_{x \rightarrow \infty} u(x)= \infty$. 
We show that $V$ represents the cost  of forcing the system \eqref{original_SDE} started on sphere
$S_{r_2}$ to leave the annulus $D_{r_1,r_3}$. 

\begin{lemma}
	\label{LDP_lower_bound}
	Let $0 \leq r_1 < r_2 < r_3 \leq\infty$ with $r_1 <1< r_3$ and let $T>0$. 
	Then,
	\begin{align*}
		\limsup_{\varepsilon \rightarrow 0} \varepsilon \log 
		\mathbb{P} \left( \tau^\varepsilon ( S_{r_2}, D_{r_1,r_3} ) \leq T \right) \leq -V(r_1,r_2,r_3)
	\end{align*}
\end{lemma}

\begin{proof}
	For any $ \phi \in C([0,T] \times \mathbb{R}^d, \mathbb{R}^d)$, 
	$x \in \mathbb{R}^d$ and $0 \leq s < t \leq T$,
	\begin{align}
		\label{potential_lower_bound}
		\begin{split}
		I_T (\phi) 
			& \geq \frac{1}{2} \int_s^t \left| \dot{\phi} (u, x) + \nabla U (\phi(u,x)) \right|^2 du \\
			& = \frac{1}{2} \int_s^t \left| \dot{\phi} (u, x) - \nabla U (\phi(u,x)) \right|^2 du
				+ 2 \int_s^t \langle \dot{\phi}(u,x), \nabla U (\phi(u,x)) \rangle du \\
			& \geq 2  \left( U ( \phi(t,x)) - U(\phi(s,x)) \right).
		\end{split}
	\end{align}
	Define
	\begin{align*}
		\Phi_i := \left\{ \phi \in C ( [0,T] \times \mathbb{R}^d, \mathbb{R}^d) : 
			\phi(0, \cdot) = Id \textrm{ and } 
			 | \phi (t,x) | = r_i \textrm{ for some } x \in S_{r_2},t \in [0,T]  \right\}
	\end{align*}
	for $i=1,3$. By LDP it follows that
	\begin{align*}
		\limsup_{\varepsilon \rightarrow 0} \varepsilon \log 
		\mathbb{P} \left( \tau^\varepsilon ( S_{r_2}, D_{r_1,r_3} ) \leq T \right)
		\leq - \inf_{\phi \in \Phi_1 \cup \Phi_3} I_T (\phi).
	\end{align*}
	We consider the case $r_2 \leq 1$. If $\phi \in \Phi_1$, there exists $x \in S_{r_2}$ and $ t \in [0,T]$
	such that $|\phi(t,x)| = r_1$. By \eqref{potential_lower_bound}, 
	$I_T(\phi) \geq 2 (U(r_1) -U(r_2))$.
	If $\phi \in \Phi_3$, there exists $x\in S_{r_2}$ and $0\leq s <t\leq T$ such that $|\phi(s,x)|=1$
	and $|\phi(t,x)|= r_3$. Using \eqref{potential_lower_bound}, it follows that
	$I_T(\phi) \geq 2 (U(r_3) -U(1))$. Repeating the same arguments for the case $r_2 >1$, the
	statement follows.
\end{proof}

\noindent
Denote by
\begin{align*}
	\sigma^\varepsilon (M,D) := \inf \left\{ t \geq 0 : X_t^{\varepsilon}(x) \in D 
	\textrm{ for all } x \in M\right\}
\end{align*}
the time until $D \subset \mathbb{R}^d$ contains the semi-flow started in $M \subset \mathbb{R}^d$.\\
The next lemma estimates the time until the semi-flow started in an annulus is contained in
a neighborhood of the stable sphere for small noise.
Observe that this time is roughly the time the semi-flow of the ODE \eqref{ODE}
started in the annulus requires to be contained in the neighborhood since the semi-flow
of the SDE \eqref{original_SDE} behaves similar to the semi-flow of the ODE \eqref{ODE}
for small noise on a fixed time scale.

\begin{lemma}
	\label{bound_on_stopping_time}
	Let $0<r_1 < r_2 <\infty$ and $0 \leq r_3 <1 < r_4 \leq \infty$. Then
	\begin{align*}
		\lim_{t \rightarrow \infty} \limsup_{\varepsilon \rightarrow 0} \varepsilon  \log 
		\mathbb{P} \left( \sigma^\varepsilon ( \overline{ D_{r_1,r_2}},  D_{r_3,r_4}) > t \right) 
		\leq - V(0, r_1, \infty).
	\end{align*}
\end{lemma}

\begin{proof}
	Set $V:=V(0, r_1, \infty)>0$ and let $0<\delta<V/2$.
	We choose $0< \alpha < r_1< r_2 < \beta$ such that 
	$V(\alpha, r_1, \beta ) \geq V - \delta /2$ and 
	$V(\alpha, r_2, \beta ) \geq V - \delta /2$.
	Set $M:= \overline{ D_{r_1,r_2}}$ and $N := \overline{ D_{\alpha,\beta}}$.
	It holds that
	\begin{align*}
		\mathbb{P} \left( \sigma^\varepsilon ( M,  D_{r_3,r_4}) > t \right) 
		\leq \mathbb{P} \left( \tau^\varepsilon (  M,  N) \leq t \right)
			+ \mathbb{P} \left( \tau^\varepsilon ( M,  N) > t
			\textrm{ and } \sigma^\varepsilon ( M,  D_{r_3,r_4}) > t \right)
	\end{align*}
	By Lemma \ref{LDP_lower_bound} there exists $\varepsilon_0>0$ such that
	\begin{align*}
		\mathbb{P} \left( \tau^\varepsilon (  M,  N) \leq t \right) 
		&\leq  \mathbb{P} \left( \tau^\varepsilon (  S_{r_1},  N) \leq t \right) 
			+ \mathbb{P} \left( \tau^\varepsilon (  S_{r_2},  N) \leq t \right) \\
		&\leq 2 \exp ( - (V - \delta) / \varepsilon )
	\end{align*}
	for all $\varepsilon \leq \varepsilon_0$.
	We consider the closed sets 
	\begin{align*}
		\Psi_t &:= \big\{ \phi \in C ( [0,t] \times \mathbb{R}^d, N ) : 
		\textrm{ for all } s \in [0,t] \textrm{ there } 
		\\ & \hspace{90pt} 
		 \textrm{exists an } x \in M \textrm{ such that } \phi(s,x) \not\in D_{r_3,r_4}
		 \big\}, \\
		\widetilde{\Psi}_t &:= \big\{ \phi \in C ( [0,t] \times \mathbb{R}^d , \mathbb{R}^d ) : 
		\textrm{ for all } s \in [0,t]
		\textrm{ there } 
		\\ & \hspace{90pt} 
		\textrm{exists an } x \in N \textrm{ such that } \phi(s,x) 
		\not\in D_{r_3,r_4} \big\}.
	\end{align*}
	The event $\left\{ \tau^\varepsilon ( M, N) > t \right\} \cap 
	\left\{ \sigma^\varepsilon ( M,  D_{r_3,r_4}) > t \right\}$ is contained in 
	$\left\{ X_t^\varepsilon \in \Psi_t \right\}$.
	By LDP 
	\begin{align*}
		\limsup_{\varepsilon \rightarrow 0} \varepsilon  \log 
		\mathbb{P} \left( \tau^\varepsilon ( M, N) > t \textrm{ and } 
		\sigma^\varepsilon ( M,  D_{r_3,r_4}) > t \right)	
		\leq - \inf_{\phi \in \Psi_t} I_t (\phi).
	\end{align*}
	It remains to show that
	\begin{align}
		\label{limit_infty}
		\lim_{t \rightarrow \infty} \inf_{\phi \in \Psi_t} I_t (\phi) > V. 
	\end{align}
	There exists an $T>0$ such that the semi-flow associated to the deterministic ODE \eqref{ODE}
	started in $N$ is in $D_{r_3 ,r_4 }$ at time $T$.
	Assume that \eqref{limit_infty} is false. Then, for  
	every $ n \in \mathbb{N}$ there exists $\phi_n  \in \Psi_{nT}$ such that 
	$I_{nT}(\phi_n) \leq V$. Hence, there exists an $ g_n \in C_0([0,nT])$ 
	with $F_{nT}(g_n)= \phi_n$ and $\hat{I}_{nT} (g_n) \leq 2V$.
	Set $g_{n,k} (t) := g_n(t+kT) - g_n(kT)$ for $0\leq k \leq n-1$ and $0\leq t\leq T$. 
	We define $\phi_{n,k}:= F_{T}( g_{n,k})$. 
	Observe that $\phi_{n,k} \in \widetilde{\Psi}_T$ since $\phi_n \in \Psi_{nT}$ and
	$\phi_n (kT +t,x) = \phi_{n,k} (t, \phi_n (kT,x))$ for all $x \in M$.
	By definition of $g_n$, it follows that
	\begin{align*}
		\sum_{k=0}^{n-1} \hat{I}_T (g_{n,k}) = \hat{I}_{nT} (g_n) \leq 2V.
	\end{align*}
	for all $n \in \mathbb{N}$.
	Hence there exists a sequence $h_n \in C_0([0,T])$ with 
	$\lim_{n \rightarrow \infty} \hat{I}_T(h_n) = 0$ and 
	$F_T(h_n) \in \tilde{\Psi}_T$ for all $n \in \mathbb{N}$.
	Arzel\`{a}-Ascoli implies that 
	$\left\{ h \in C_0([0,T]) :  \hat{I}_T (h) \leq 2V \right\}$
	is a compact subset of $ C_0([0,T])$. 
	Therefore, the sequence $h_n$ has a limit point $h$ in $C_0([0,T])$. 
	Continuity of $F_T$ implies that $\psi := F_T(h) \in \tilde{\Psi}_T$.
	By lower semi-continuity of $\hat{I}_T$, 
	$I_T(\psi) =0$ and $\psi$ describes the flow of the deterministic
	ODE \eqref{ODE}. 
	By definition of $T$, for all $x\in N$ it holds that $\psi(T,x) \in D_{r_3,r_4}$ which is a contradiction
	to $\psi \in \widetilde{\Psi}_T$.	
\end{proof}

\begin{proposition}
	\label{prop_lower_bound}
	Let $0 \leq  r_1 < r_2 < r_3 \leq \infty$ with $r_1 < 1 <r_3$. 
	Set $V:=V(r_1,1,r_3)$.
	For any $\beta>0$ it holds that
	\begin{align*}
		\lim_{\varepsilon \rightarrow 0}  \mathbb{P} 
		\left( \tau^\varepsilon (S_{r_2}, D_{r_1,r_3}) > \exp ( (V - \beta )/ \varepsilon ) \right) =1
	\end{align*}
	and 
	\begin{align*}
		\lim_{\varepsilon \rightarrow 0} \varepsilon \log \mathbb{E} \tau^\varepsilon (S_{r_2}, D_{r_1,r_3}) 
		\geq V.
	\end{align*}
\end{proposition}

\begin{proof}
	Let $\beta < V(r_1,r_2,r_3)$ and
	$\eta >0$ be small enough such that $ r_1 < 1-2 \eta$, $r_3 > 1 + 2 \eta$, 
	$V(r_1, 1-2 \eta, r_3) > V - \beta / 4$ and $V(r_1, 1+2 \eta, r_3) > V - \beta / 4$.
	Let  $\rho_0 =0$ and for $n \in \mathbb{N}_0$ define the stopping times
	\begin{align*}
		\sigma_n &:= 
		\inf \big\{ t \geq \rho_n : 
			\left| X_t^\varepsilon (x) \right| \in (1- \eta, 1+\eta) \textrm{ for all } 
			x \in S_{r_2} \\
			& \qquad \qquad \qquad \qquad \textrm{ or }
			\left| X_t^\varepsilon (x) \right| \not\in (r_1, r_3) \textrm{ for some } x \in S_{r_2}
			\big\}, \\
		\rho_{n+1} &:= \inf \left\{ t \geq \sigma_n : 
			\left| X_t^\varepsilon (x) \right| \not\in (1- 2 \eta, 1+ 2 \eta) 
			\textrm{ for some } x \in S_{r_2} \right\}
	\end{align*}
	with convention that $\rho_{n+1} = \infty $ if 
	$ \sigma_n = \tau^\varepsilon (S_{r_2}, D_{r_1,r_3})$.
	During each time interval $[ \rho_n, \sigma_n]$ one point of the semi-flow either leaves the annulus
	$D_{r_1, r_3}$ or the semi-flow reenters the smaller annulus  $D_{1-\eta, 1+\eta}$. 
	Note that necessarily $\tau^\varepsilon (S_{r_2}, D_{r_1,r_3})= \sigma_n$ for some $n \in \mathbb{N}_0$.
	\begin{figure}[h]
		\centering
		\includegraphics[width=0.8\textwidth]{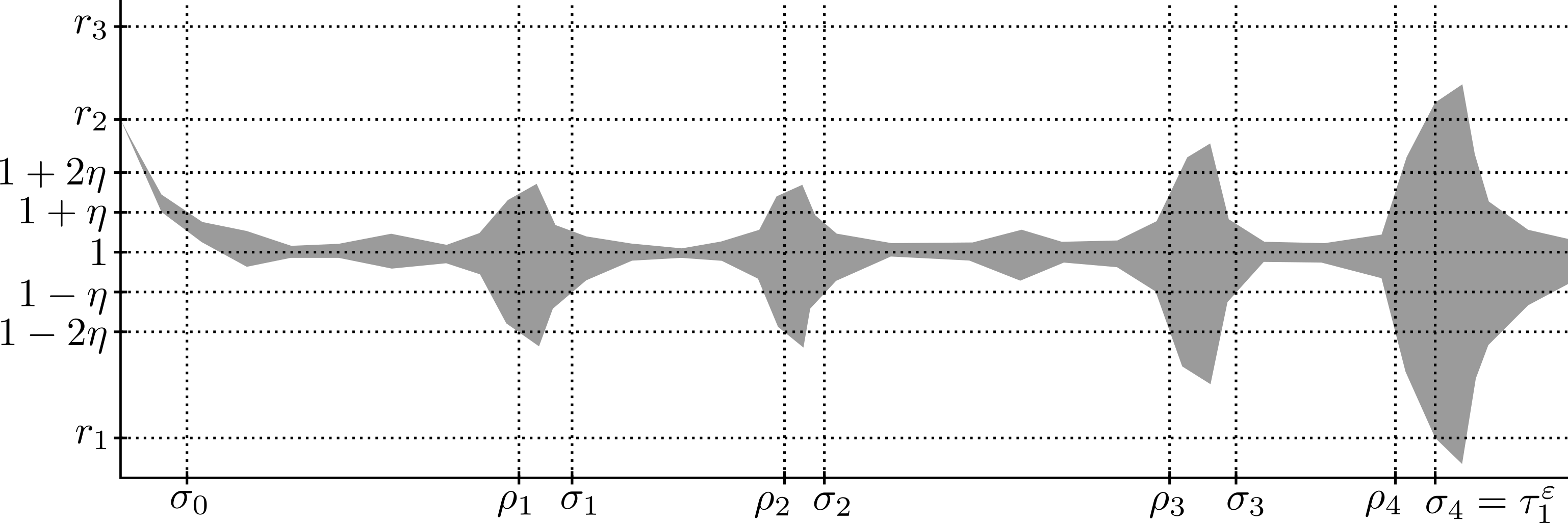}
		\caption{Outline of the set $|X_t^\varepsilon(S_{r_2})|$ and the stopping times $\sigma_n$ and $\rho_n$}
	\end{figure}
	\\
	By Lemma \ref{bound_on_stopping_time} there exists an $T >0$ and $\varepsilon_1>0$ such that
	\begin{align*}
		\mathbb{P} \left( \sigma_0 >T \right) &\leq 
		\mathbb{P} \left( 
			\sigma^\varepsilon ( S_{r_2}, D_{1-\eta,1+\eta}) > T \right) \\
		&\leq \exp(-(V(r_1,r_2,r_3)-\beta)/\varepsilon)
	\end{align*}
	and
	\begin{align*}
		\mathbb{P} \left( \sigma_n - \rho_n >T \right) 
		&\leq \mathbb{P} \left( \sigma^\varepsilon ( \overline{D_{1-2\eta,1+2\eta}},
			 D_{1-\eta,1+\eta}) 
			>T \right) \\
		&\leq \exp(-(V-\beta/2)/\varepsilon) 
	\end{align*}
	for all $n \in \mathbb{N}$ and $\varepsilon \leq \varepsilon_1$.
	Using Lemma \ref{LDP_lower_bound}, there exists $\varepsilon_2>0$ such that
	\begin{align}
		\label{probability_sigma_0}
		\begin{split}
		\mathbb{P} ( \tau^\varepsilon (S_{r_2}, D_{r_1,r_3}) = \sigma_0 ) 
		&\leq \mathbb{P} \left( \sigma_0  >T \right) + 
			\mathbb{P} (\tau^\varepsilon (S_{r_2}, D_{r_1,r_3})  \leq T) \\
		& \leq  2 \exp(-(V(r_1,r_2,r_3)-\beta)/\varepsilon)
		\end{split}
	\end{align}
	and 
	\begin{align}
		\label{probability_sigma_n}
		\begin{split}
		\mathbb{P} ( \tau^\varepsilon (S_{r_2}, D_{r_1,r_3}) = \sigma_n ) 
		&\leq \mathbb{P} \left( \sigma_n -\rho_n  >T \right) + 
			\mathbb{P} (\tau^\varepsilon (S_{1-2\eta}, D_{r_1,r_3})  \leq T) \\
		& \quad	 + \mathbb{P} (\tau^\varepsilon (S_{1+2\eta}, D_{r_1,r_3})  \leq T)\\
		& \leq 3  \exp(-(V-\beta/2)/\varepsilon)
		\end{split}
	\end{align}
	for all $n \in \mathbb{N}$ and $ \varepsilon \leq \varepsilon_2$. Choose $T_0>0$ such that 
	$2d T_0 (V -\beta/2) \leq \eta^2$. 
	Then, for all $n \in \mathbb{N}$
	\begin{align}
		\label{estimate_length_interval}
		\begin{split}
		\mathbb{P} \left( \rho_n - \sigma_{n-1} \leq T_0 \right) 
		&\leq \mathbb{P} \left( \sup_{t \in [0,T_0]} \sqrt{\varepsilon} \left| W_t \right| \geq \eta \right) 
		\leq 4d \exp ( - \eta^2 / (2d T_0 \varepsilon)) \\
		&\leq 4d \exp(-(V-\beta/2)/\varepsilon).
		\end{split}
	\end{align}
	The event $\left\{ \tau^\varepsilon (S_{r_2}, D_{r_1,r_3}) \leq kT_0 \right\}$ implies that either
	$\left\{ \tau^\varepsilon (S_{r_2}, D_{r_1,r_3}) = \sigma_n \right\}$ for some $0 \leq n \leq k$ or
	that at least one of the interval $[\sigma_n, \sigma_{n+1}]$ for $0 \leq n < k$ is at most of length $T_0$.
	Combining the estimates \eqref{probability_sigma_n} and 
	\eqref{estimate_length_interval}, it follows that
	\begin{align*}
		\mathbb{P} \left( \tau^\varepsilon (S_{r_2}, D_{r_1,r_3}) \leq kT_0 \right)
		&\leq \sum_{n=0}^{k} \mathbb{P} \left( \tau^\varepsilon (S_{r_2}, D_{r_1,r_3}) = \sigma_n  \right)
			+ \sum_{n=1}^{k} \mathbb{P} \left( \rho_n - \sigma_{n-1} \leq T_0 \right) \\
		&\leq \mathbb{P} ( \tau^\varepsilon (S_{r_2}, D_{r_1,r_3}) = \sigma_0 ) +
			(3+4d)k \exp(-(V-\beta/2)/\varepsilon) 	
	\end{align*}
	for all $k \in \mathbb{N}$ and 
	$\varepsilon \leq \varepsilon_0 := \min \left\{ \varepsilon_1, \varepsilon_2 \right\}$.
	Choose $k$ to be $T_0^{-1} \exp((V-\beta)/\varepsilon)$ rounded up to integers. Hence,
	\begin{align*}
		\mathbb{P} \left( \tau^\varepsilon (S_{r_2}, D_{r_1,r_3}) \leq \exp((V-\beta)/\varepsilon) \right)
		&\leq \mathbb{P} \left( \tau^\varepsilon (S_{r_2}, D_{r_1,r_3}) \leq kT_0 \right) \\
		& \leq \mathbb{P} ( \tau^\varepsilon (S_{r_2}, D_{r_1,r_3}) = \sigma_0 ) + 
			8d T_0^{-1} \exp(-\beta/ (2\varepsilon)) 
	\end{align*}
	for small enough $\varepsilon$. By estimate \eqref{probability_sigma_0}, the right side of the
	inequality converges to zero as $\varepsilon \rightarrow 0$. 
	The lower bound for $\mathbb{E} \tau^\varepsilon (S_{r_2}, D_{r_1,r_3}) $ follows by Markov's inequality.
\end{proof}

\begin{corollary}
	\label{cor_lower_bound}
	Set $V:=V(0,1,\infty)$. For any $\beta >0$ there exists $\delta_0>0$ such that
	\begin{align*}
		\lim_{\varepsilon \rightarrow 0}  \mathbb{P} 
		\left( \tau^\varepsilon_{1,\delta} > \exp ( (V - \beta )/ \varepsilon ) \right) =1
	\end{align*}
	and 
	\begin{align*}
		\lim_{\varepsilon \rightarrow 0} \varepsilon \log \mathbb{E} \tau^\varepsilon_{1,\delta}
		 \geq V
	\end{align*}
	for any $0<\delta<\delta_0$.
\end{corollary}

\begin{proof}
	Observe that $ \tau_{1, \delta}^\varepsilon \geq \tau^\varepsilon (S_1, D_{\delta /2 , \infty})$.
\end{proof}

\subsection{Upper bound for $\tau_{3,\delta}^\varepsilon$}
\label{section_upper_bound}

In this subsection, we give an upper bound for $\tau^\varepsilon_{3, \delta}$.
Here, $\tau^\varepsilon_{3, \delta}$ is
associated to the solution of \eqref{original_SDE}
where the differential equation \eqref{original_SDE} is additionally assumed to decay strongly.
Since $X_t^\varepsilon$ satisfies the LDP, it is sufficient to choose a sample path to get a lower
estimate on the probability that $\tau^\varepsilon_{3, \delta}$ is smaller than some fixed time. 
Using this probability as the success probability of a geometric distribution, we get the upper bound 
for $\tau^\varepsilon_{3, \delta}$. 

\begin{lemma}
	\label{LDP_upper_bound}
	Assume that the SDE \eqref{original_SDE} is strongly contracting.
	For any $\delta>0$ 
	\begin{align*}
	 	\lim_{T \rightarrow \infty}	\liminf_{\varepsilon \rightarrow 0} \varepsilon \log 
		\mathbb{P} \left( \tau_{3,\delta}^\varepsilon \leq T \right) 
		\geq -V(0,1,\infty).
	\end{align*}
\end{lemma}

\begin{proof}
	Denote by $u'$ the first derivative of $u$.
	Let $0< \alpha <1$ be small enough such that $u'(4)\geq 2 \alpha$.  Set
	$c_1^\alpha	:= \max \left\{ 1- \alpha, \sup \left\{ 0<x<1 : u'(|x|^2) \leq -\alpha \right\} \right\}$
	and $c_2^\alpha := \inf \left\{ x>1 : u'(|x|^2) \geq 2 \alpha \right\} \leq 2$.\\
	We choose $g^\alpha (t):= (\int_0^t h^\alpha (s) \, ds , 0 , \dots, 0) \in \mathbb{R}^d$
	with
	\begin{align*}
		h^\alpha (s) :=
		\begin{cases}
			0, & 
				\textrm{for } 0 \leq s \leq T_1^\alpha \textrm{ or } T_5^\alpha < s \leq T_6^\alpha \\
			3 \alpha, & 
				\textrm{for } T_1^\alpha < s \leq T_2^\alpha  \\
			2 \nabla \tilde{U}(\varphi(s- T_2^\alpha)), &
				\textrm{for } T_2^\alpha < s \leq T_3^\alpha \\
			(- 2 u'(0) +1) \alpha, & 
				\textrm{for } T_3^\alpha < s \leq T_4^\alpha \\
			\beta^\alpha, & 
				\textrm{for } T_4^\alpha < s \leq T_5^\alpha \\
			4 \alpha c_2^\alpha, &
				\textrm{for } T_6^\alpha < s \leq T_7^\alpha 
		\end{cases}
	\end{align*}
	for some $\beta^\alpha>0$, $0< T_1^\alpha< T_2^\alpha< \dots < T_6^\alpha < \infty$
	determined in the following and
	where $\varphi$ is the solution of
	\begin{align*}
		\dot{\varphi} (s) = \nabla \tilde{U}(\varphi(s)) \quad \textrm{on } \mathbb{R}
	\end{align*}
	started in $\varphi(0)= -c_1^\alpha$ where $\tilde{U}(x):= u(x^2)$. Hence, 
	\begin{align*}
		\hat{I}_{T_3^\alpha -T_2^\alpha} (g^\alpha (\cdot + T_2^\alpha)) &= 
		2 \int_{0}^{T_3^\alpha-T_2^\alpha} \langle \dot{\varphi} (s) ,  \nabla \tilde{U} (\varphi (s)) \rangle ds
		= 2 ( \tilde{U} (\varphi(T_3^\alpha-T_2^\alpha)) -\tilde{U}(\varphi(0)) ) \\
		&\leq 2 (u(0)-u(1))= V(0,1,\infty).
	\end{align*}
	Moreover, $\hat{I}_{T_{j+1}^\alpha -T_j^\alpha} (g^\alpha (\cdot + T_j^\alpha)) = 0$ for $j=0,4$ and
	\begin{align*}
		 \hat{I}_{T_{j+1}^\alpha -T_j^\alpha} (g^\alpha(\cdot + T_j^\alpha )) 
		= \int_{T_j^\alpha}^{T_{j+1}^\alpha} \left| h^\alpha (s) \right|^2 ds
		\leq (h^\alpha (T_{j+1}^\alpha))^2 \left(T_{j+1}^\alpha -T_j^\alpha \right)
	\end{align*}
	for $j=1,3,4,6$.
	Denote by $F(g) := F_{T_7^\alpha}(g)$ the semi-flow associated
	to \eqref{semi-flow_with_g}. \\
	In the following, we choose $\beta^\alpha$ and $T_i^\alpha$ for $i=1,2,...,7$ such that
	\begin{align*}
		\lim_{\alpha \rightarrow 0} (h^\alpha (T_{j+1}^\alpha))^2 \left(T_{j+1}^\alpha -T_j^\alpha \right) =0
	\end{align*}
	for $j=1,3,4,6$ and 
	\begin{align*}
		\left| F(g^\alpha) (T_7^\alpha,x) - F(g^\alpha) (T_7^\alpha,y) \right| \leq \delta
	\end{align*}
	for all $x,y \in \mathbb{R}^d$. Then
	\begin{align*}
		\lim_{\alpha \rightarrow 0} \liminf_{\varepsilon \rightarrow 0} \varepsilon \log \mathbb{P} 
		\left( \tau_{3,\delta}^\varepsilon \leq T_7^\alpha \right) 
		&\geq - \lim_{\alpha \rightarrow 0} \hat{I}_{T_7^\alpha} (g^\alpha) 
		= - \lim_{\alpha \rightarrow 0} \sum_{j=0}^7 
			\hat{I}_{T_{j+1}^\alpha -T_j^\alpha} (g^\alpha(\cdot + T_j^\alpha )) \\
		&\geq - V(0,1,\infty)
	\end{align*}
	by LDP and the statement follows. 
	\begin{figure}[h]
		\centering
		\begin{subfigure}[c]{0.19\textwidth}
			\centering
			\includegraphics[width=0.8\textwidth]{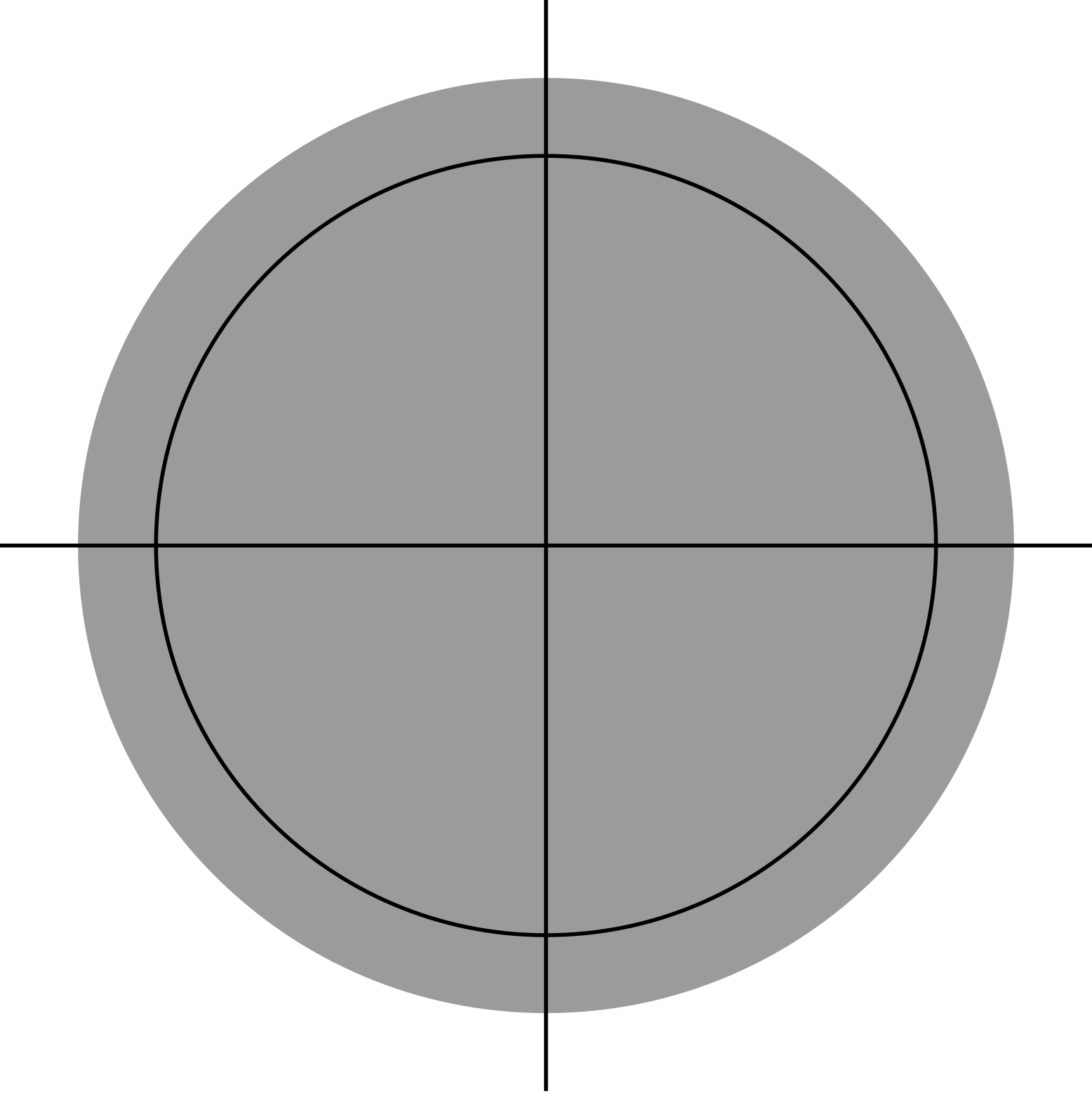}
			\subcaption{t=$T_1^\alpha$}
		\end{subfigure}
		\begin{subfigure}[c]{0.19\textwidth}
			\centering
			\includegraphics[width=0.8\textwidth]{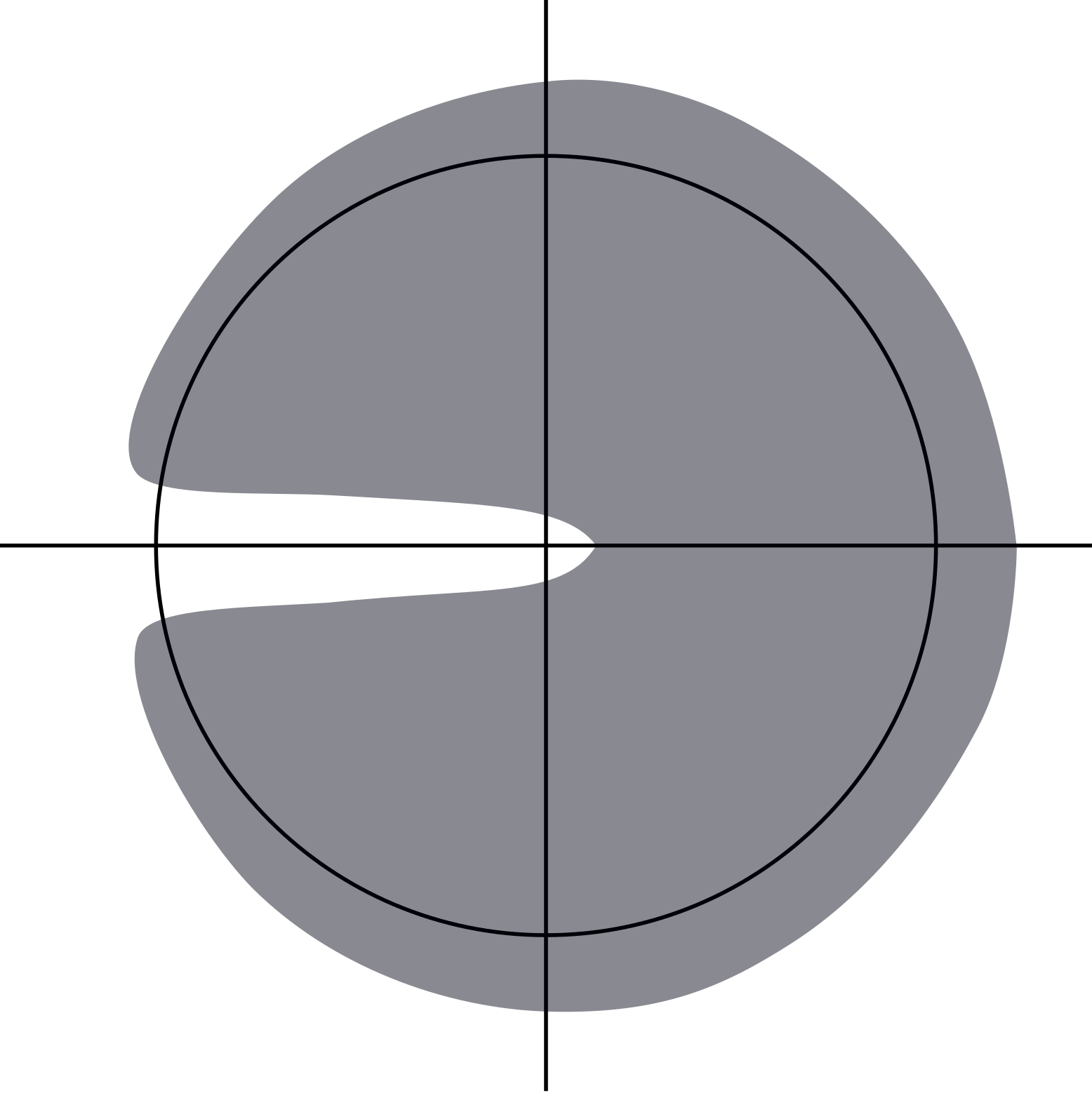}
			\subcaption{t=$T_4^\alpha$}
		\end{subfigure}
		\begin{subfigure}[c]{0.19\textwidth}
			\centering
			\includegraphics[width=0.8\textwidth]{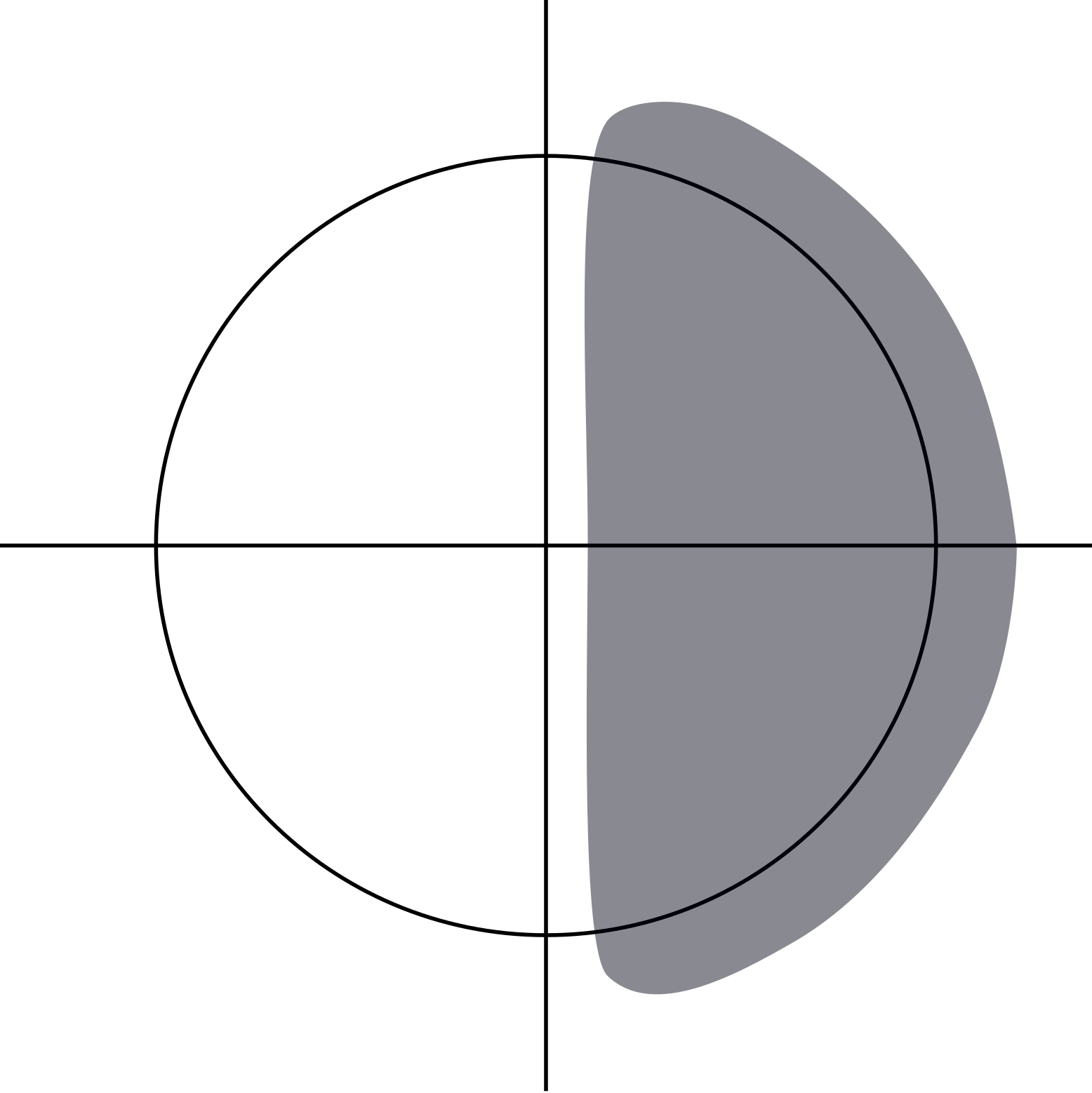}
			\subcaption{t=$T_5^\alpha$}
		\end{subfigure}
		\begin{subfigure}[c]{0.19\textwidth}
			\centering
			\includegraphics[width=0.8\textwidth]{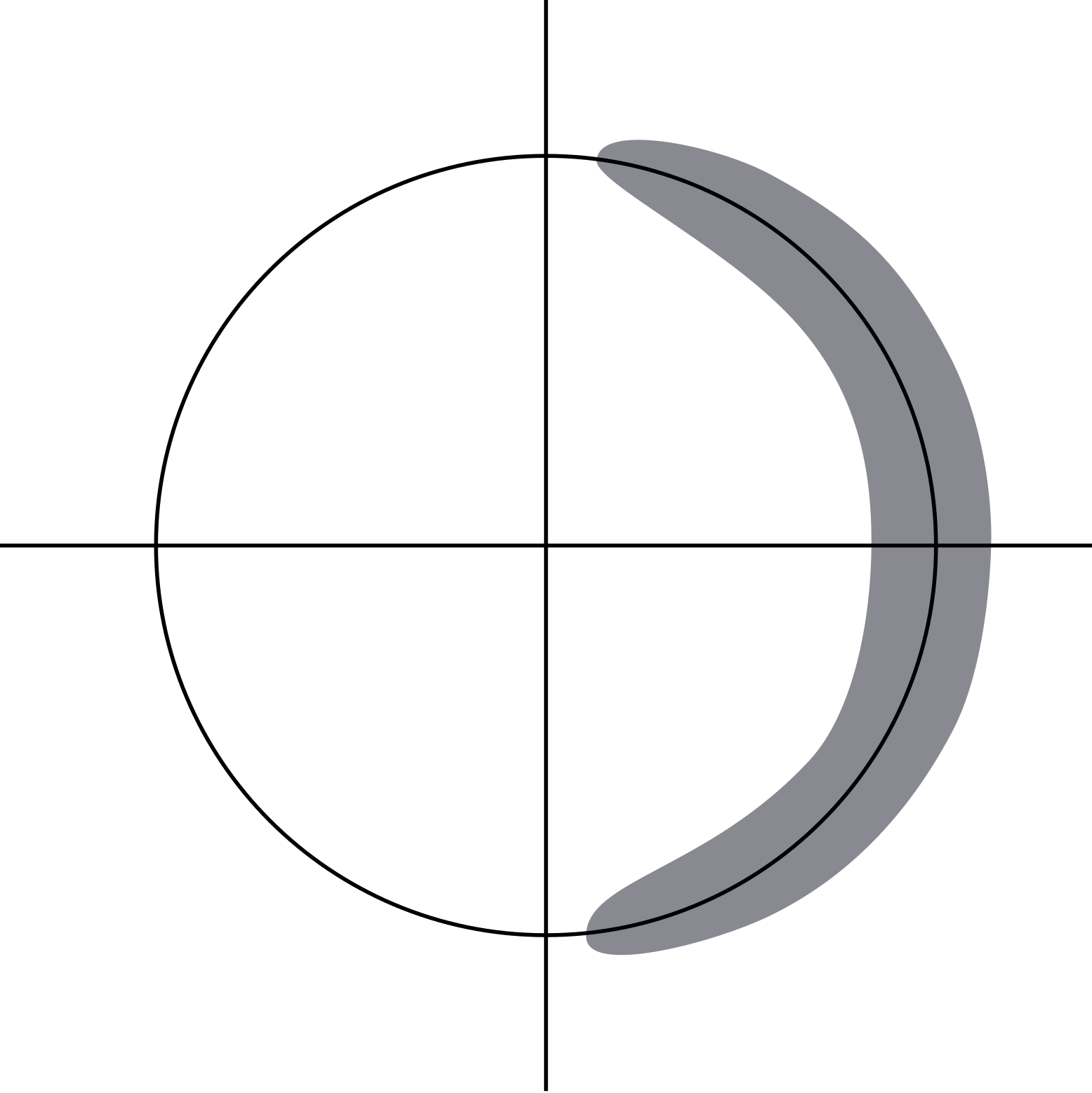}
			\subcaption{t=$T_6^\alpha$}
		\end{subfigure}
		\begin{subfigure}[c]{0.19\textwidth}
			\centering
			\includegraphics[width=0.8\textwidth]{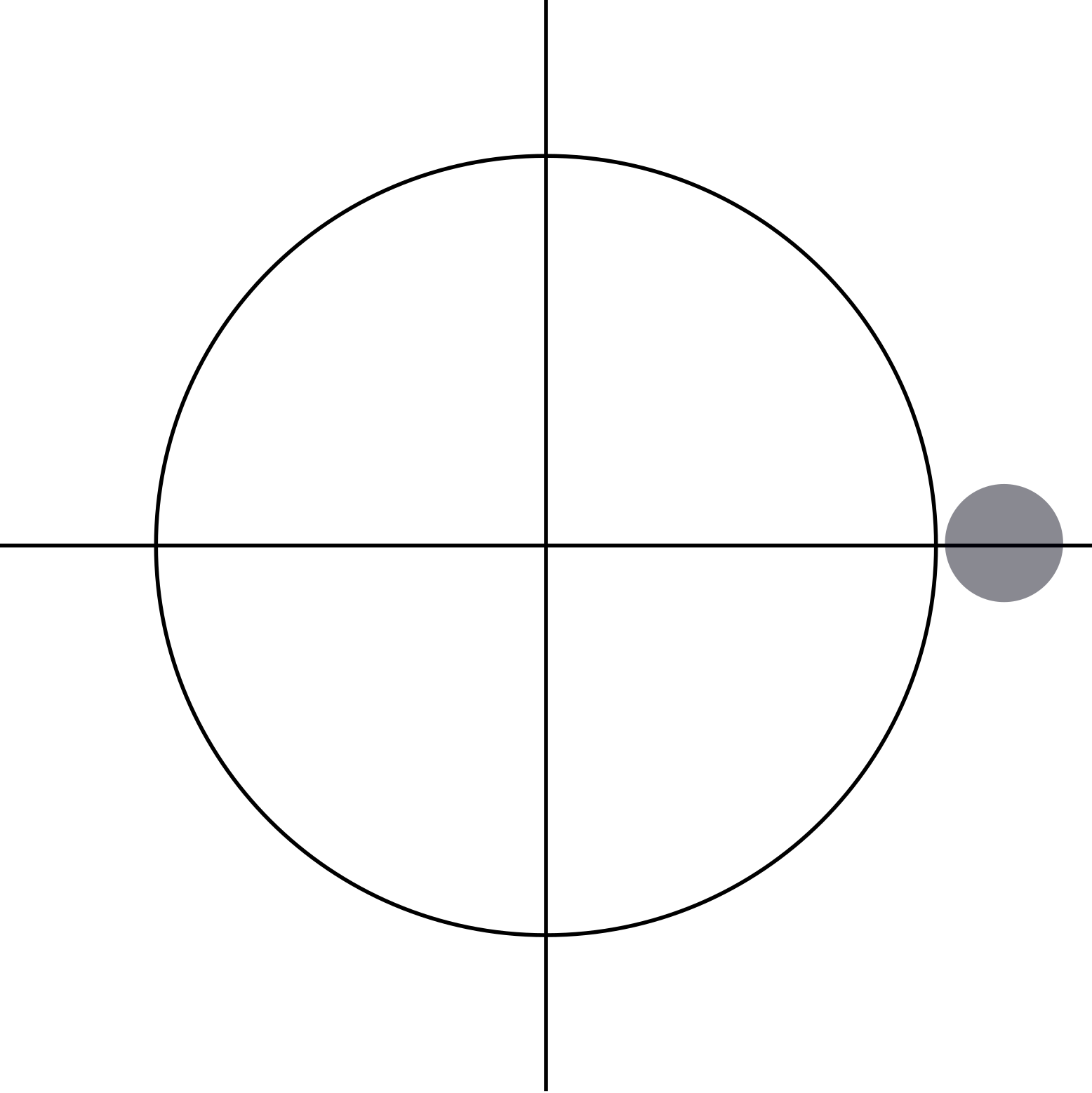}
			\subcaption{t=$T_7^\alpha$}
		\end{subfigure}
		\caption{Outline of the semi-flow $F(g^\alpha)$ in $\mathbb{R}^2$ at time $t$}
	\end{figure}
	\\
	\emph{Step 1:}
	Since \eqref{original_SDE} is strongly contracting, we can choose $T_1^\alpha$ 
	such that $\left| F(g^\alpha) (T_1^\alpha,x) \right| \leq 1 +  \alpha$ for
	all $x \in \mathbb{R}^d$. 
	\\
	Define $Y(t,y) := F(g^\alpha( t + T_1^\alpha)) (t,y)$ for $y \in \mathbb{R}^d$ 
	and write $t_k^\alpha := T_k^\alpha - T_1^\alpha$ for $k=2, \dots,7$.
	Observe that it is sufficient to restrict the analysis to $Y(t,y)$ on the set $\bar{B}_{1+\alpha}$
	since $Y(t,y)$ describes the dynamics of $F(g^\alpha)$ after time $T_1^\alpha$.
	Denote by $\Pi_k$ the projection on the $k$-th component in $\mathbb{R}^d$.
	\\
	In the steps 2 to 4, we concentrate on the movement of the point 
	$y_1:= (-1- \alpha,0, \dots,0)\in \mathbb{R}^d$, choose $t_2^\alpha, t_3^\alpha$ and $t_4^\alpha$ 
	and show that $\Pi_1 ( Y(t_4^\alpha,y_1)>0$. 
	This behavior we extend to the set $\bar{B}_{1+\alpha}$ in step 5 by choosing $\beta^\alpha$ and 
	$t_5^\alpha$ suitable and showing that $\Pi_1 ( Y(t_5^\alpha, y))>0$ for all $y \in \bar{B}_{1+\alpha}$.
	Observe that $\Pi_1 ( Y(s, y))>0$ implies that $\Pi_1 ( Y(t, y))>0$ for 
	all $y \in \mathbb{R}^d$ and $0<s<t$ 
	and define $t_y := \inf \left\{ t \geq 0: \Pi_1(Y(t,y))>0  \right\}$ for $ y \in \mathbb{R}^d$.
	Hence, $\Pi_1 ( Y(t, y))>0$ for all $t \geq t_y$.
	In the steps 6 and 7, we choose $t_6^\alpha$ and $t_7^\alpha$ and show the contraction. \\
	\emph{Step 2:}
	Set $y_1:= (-1 - \alpha, 0, \dots,0)\in \mathbb{R}^d$.
	Observe that $\Pi_1(F(t \mapsto 3 \alpha t) (2,y_1)) \geq -c_1^\alpha$. Choose $t_2^\alpha= 2$. \\
	\emph{Step 3:}
	The function $\varphi$ as defined above describes the movement of 
	$Y(\cdot + t_2^\alpha,y_2)$ started in 
	$y_2:=(- c_1^\alpha,0,\dots,0)\in \mathbb{R}^d$. Choose $t_3^\alpha$ such that 
	$\varphi ( t_3^\alpha - t_2^\alpha) \geq - \alpha$. Then,
	$ \Pi_1 (Y(t_3^\alpha ,y_2)) \geq - \alpha$. \\
	\emph{Step 4:}
	Let $y_3:= (-\alpha,0, \dots,0)\in \mathbb{R}^d$.
	Observe that $\Pi_1 (F(t \mapsto (-2u'(0) +1) \alpha t) (2,y_3) ) \geq \alpha$.
	Choose $t_4^\alpha= t_3^\alpha +2$. Then, $ \Pi_1(Y (t_4^\alpha,y_1)) \geq \alpha>0$.\\
	\emph{Step 5:}
	Since $y \mapsto Y (t_4^\alpha,y)$ is continuous, there exists a neighborhood of $y_1$
	such that $\Pi_1(Y (t_4^\alpha,y)) >0$ for all $y$ in this neighborhood.
	Hence, there exists an $\eta^\alpha>0$ such that $\Pi_1(Y (t_4^\alpha,y)) <0$ for
	some $y \in S_{1+\alpha}$ implies that $\Pi (y) \geq -1 -\alpha + \eta^\alpha$.
	Observe that 
	\begin{align*}
		d \; \frac{\Pi_1(F(g^\alpha) (t,y)) }{\sqrt{\sum_{k=2}^d (\Pi_k(F(g^\alpha) (t,y)))^2 }} 
		= \frac{1 }{\sqrt{\sum_{k=2}^d (\Pi_k(F(g^\alpha) (t,y)))^2 }} \; dg^\alpha(t)
	\end{align*}
	for all $y \in \mathbb{R}^d$. 
	Observe that $|Y(t,y)| \leq 2$ for all $y \in \bar{B}_{1+\alpha}$ and $t_4^\alpha \leq t < t_y$.
	Hence, 
	\begin{align*}
		\frac{\Pi_1(Y (t,y)) }{\sqrt{\sum_{k=2}^d (\Pi_k(Y (t,y)))^2 }} 
		\geq - \frac{2}{\eta^\alpha} +
			\frac{1 }{2} \beta^\alpha (t - t_4^\alpha)
	\end{align*}
	for all $y \in S_{1+\alpha}$ and $t_4^\alpha \leq t < t_y$. 
	Set $ \beta^\alpha := \alpha  \eta^\alpha $ and 
	$ t_5^\alpha := t_4^\alpha + 8 \alpha^{-1} (\eta^\alpha)^{-2}$.
	Then, $\Pi_1(Y (t_5^\alpha,y))>0$ for any $y \in S_{1+\alpha}$.
	Moreover, 
	\begin{align*}
	 	\hat{I}_{T_{5}^\alpha -T_4^\alpha} (g^\alpha(\cdot + T_4^\alpha )) 
		= (h^\alpha (T_{5}^\alpha))^2 \left(T_{5}^\alpha -T_4^\alpha \right) 
		= (\beta^\alpha)^2 \left(t_{5}^\alpha -t_4^\alpha \right)
		=  8\alpha.
	\end{align*}
	\emph{Step 6:}
	Since $\Pi_1(Y (t_5^\alpha,y))>0$ for any $y \in \bar{B}_{1+\alpha}$ and 
	$\bar{B}_{1+\alpha}$ is closed, it follows that
	$ \min_{y \in \bar{B}_{1+\alpha}} \Pi_1(Y (t_5^\alpha,y))>0$.
	Choose $t_6^\alpha > t_5^\alpha$ 
	such that $ c_1^\alpha \leq \left| Y (t_6^\alpha,y) \right| \leq 2$
	for all $y \in \bar{B}_{1+\alpha}$. \\
	\emph{Step 7:}
	Since $ \Pi_1(Y (t_6^\alpha,x)) >0$, 
	$\left|Y (t_6^\alpha,x) \right| \geq c_1^\alpha$ and $ \Pi_1(g^\alpha(t)) \geq 0$ 
	for all $x \in \bar{B}_{1+\alpha}$ and $ t\geq 0$, 
	it holds that
	$\left| Y (t,x) \right| \geq c_1^\alpha $ for all $x \in\bar{B}_{1+\alpha}$ and 
	$ t_6^\alpha \leq t \leq t_7^\alpha$.
	Observe that $z:=(c_2^\alpha, 0, \dots ,0) \in \mathbb{R}^d$ is a fixed point 
	of $Y(t+t_6^\alpha, \cdot)$.
	By convexity of $u$,
	for any $x=(x_1,x_2, \dots, x_d) \in \mathbb{R}^d$ with $x_1>0$ and $|x|\geq c_1^\alpha$
	\begin{align*}
		d \,\left| Y(t + t_6^\alpha ,x) - z \right|^2
		&\leq - \frac{1}{2} \left| Y(t + t_6^\alpha ,x) - z \right|^2 \, 
			\left( u'(| Y(t + t_6^\alpha ,x) |^2) + u'(|z|^2)) \right) \; dt \\
		& \leq - \frac{\alpha}{2}  \left|Y(t + t_6^\alpha,x) - z \right|^2 \; dt.
	\end{align*}
	By Gronwall's inequality, it follows that
	\begin{align*}
		\left| Y(t + t_6^\alpha ,x) - z \right|
		\leq 4 \, \exp \left( - \frac{\alpha t}{4} \right).
	\end{align*}
	Choose $t_7^\alpha = t_6^\alpha + \frac{4}{\alpha} ( \log 8 -\log \delta )$.
	Combining all steps, it follows that 
	$ \left| F(g^\alpha) (T_7^\alpha,x) - F(g^\alpha) (T_7^\alpha,y) \right| \leq \delta $ 
	for all $x,y \in \mathbb{R}^d$.
\end{proof}

\begin{proposition}
	\label{upper_bound}
	Assume that the SDE \eqref{original_SDE} is strongly contracting.
	Set $V:= V(0,1,\infty)$.
	Then, for any $\delta>0$ and $\beta>0$ it holds that
	\begin{align*}
		\lim_{\varepsilon \rightarrow 0}  \mathbb{P} 
		\left( \tau_{3,\delta}^\varepsilon < \exp ( (V + \beta )/ \varepsilon ) \right) =1
	\end{align*}
	and 
	\begin{align*}
		\lim_{\varepsilon \rightarrow 0} \varepsilon \log \mathbb{E} \tau_{3,\delta}^\varepsilon \leq V.
	\end{align*}
\end{proposition}

\begin{proof}
	Let $0< \eta < \beta/2$.
	By Lemma \ref{LDP_upper_bound} there exists $\varepsilon_0>0$ and $T>0$ such that
	\begin{align*}
		\mathbb{P} \left( \tau_{3,\delta} \leq T \right) \geq \exp ((-V - \eta)/ \varepsilon).
	\end{align*}
	for all $\varepsilon \leq \varepsilon_0$. 
	Conditioning on the event $\left\{ \tau_{3,\delta}^\varepsilon >(k-1)T \right\}$ for $k=2,3,...$ yields 
	\begin{align*}
		\mathbb{P} \left( \tau_{3,\delta}^\varepsilon > k T \right)
		&= \mathbb{P} \left( \tau_{3,\delta}^\varepsilon > k T | \tau_{3,\delta}^\varepsilon > (k-1) T  \right)
			\mathbb{P} \left( \tau_{3,\delta}^\varepsilon > (k-1) T \right) \\
		&\leq \mathbb{P} \left( \tau_{3,\delta}^\varepsilon > T  \right) 
			\mathbb{P} \left( \tau_{3,\delta}^\varepsilon > (k-1) T \right)\\
		& \leq \mathbb{P} \left( \tau_{3,\delta}^\varepsilon > T  \right)^k
	\end{align*}
	Therefore,
	\begin{align*}
		\mathbb{E} \tau_{3,\delta}^\varepsilon 
			&\leq T \left( 1 + \sum_{k=1}^\infty \mathbb{P} 
				\left( \tau_{3,\delta}^\varepsilon > k T \right) \right)
			\leq T \left( 1 + \sum_{k=1}^\infty 
				\left( 1 - \exp ((-V - \eta)/ \varepsilon) \right)^k \right) \\
			&\leq T \exp ((V + \eta)/ \varepsilon)
	\end{align*}
	for all $\varepsilon \leq \varepsilon_0$. Using Markov's inequality it follows that
	\begin{align*}
		\mathbb{P} \left( \tau_{3,\delta}^\varepsilon  \geq  \exp ((V + \beta)/ \varepsilon) \right)
		\leq T \exp( - \beta / (2\varepsilon))
	\end{align*}
	for all $\varepsilon \leq \varepsilon_0$.
\end{proof}

\begin{remark}
	Observe that the upper bound for $\tau_{3,\delta}^\varepsilon$
	as in Proposition \ref{upper_bound}
	even hold for some RDS that do not synchronize. \\
	In \cite{Vorkastner2018}, an example of a SDE is presented which does not synchronize for small noise.
	The drift of this SDE is of the same form as in the SDE \eqref{original_SDE} while the 
	noise merely acts in the first component.
	Hence, the arguments in Lemma \ref{LDP_upper_bound} and Proposition \ref{upper_bound} extend
	to this SDE	since $g^\alpha$ in Lemma \ref{LDP_upper_bound} 
	is chosen to be $0$ in all components except for the first one. 
\end{remark}

\subsection{Approaching the set attractor}

Combining the estimates from the previous subsections, we get lower and upper bounds 
for these stopping times. 
These bounds show that the time a set requires to approach the
attractor is roughly $\exp(V(0,1,\infty)/\varepsilon)$. 

\begin{theorem}
	Assume that the SDE \eqref{original_SDE} is strongly contracting.
	Set $V:= V(0,1,\infty)$ and let $S_1 \subset M \subset \mathbb{R}^d$. For any $\beta>0$ there exists 
	$\delta_0>0$ such that for all $0 < \delta \leq \delta_0$ it holds that
	\begin{align*}
		\lim_{\varepsilon \rightarrow 0}  \mathbb{P} 
		\left( \exp ( (V - \beta )/ \varepsilon ) < 
		\tau_{1,2\delta}^\varepsilon \leq \tau_{2,\delta,M}^\varepsilon \leq \tau_{3,\delta}^\varepsilon 
		< \exp ( (V + \beta )/ \varepsilon ) \right) =1
	\end{align*}
	and 
	\begin{align*}
		\lim_{\varepsilon \rightarrow 0} \varepsilon \log \mathbb{E} \tau^\varepsilon_{1,2\delta} =
		\lim_{\varepsilon \rightarrow 0} \varepsilon \log \mathbb{E} \tau_{2,\delta,M}^\varepsilon =
		\lim_{\varepsilon \rightarrow 0} \varepsilon \log \mathbb{E} \tau^\varepsilon_{3,\delta} = V.
	\end{align*}
\end{theorem}

\begin{proof}
	Using $\tau_{1,2\delta}^\varepsilon \leq \tau_{2,\delta,M}^\varepsilon \leq \tau_{3,\delta}^\varepsilon$,
	the statement follows by Corollary \ref{cor_lower_bound} and Proposition \ref{upper_bound}.
\end{proof}



\section{Time required for a point to approach the attractor}
\label{section_weak_syn}

\subsection{Convergence to a process on the unit sphere}

In this section, we show that the time required for a point to approach the attractor under the 
dynamics of \eqref{original_SDE} in dimension $d=2$ is  exactly of order $\varepsilon^{-1}$. 
In particular, we give an estimate on the rate of convergence of a point under the dynamics of
\eqref{original_SDE} towards the attractor. 
\\
Here, we consider the minimal weak point attractor $A^{X,\varepsilon}_{point}$. 
A minimal weak point attractor is a weak point attractor
that is contained in any other weak point attractor. 
By \cite[Theorem 3.1]{Dimitroff2011} and \cite[Theorem 23]{Crauel2018} 
such a minimal weak point attractor exists.
\\
In dimension $d=1$, the time until two points approach each other is the same as the time until
the diameter of the of the interval between both points to get small. 
Hence, in dimension $d=1$ the time until a point to approach the attractor can be described by methods 
of section \ref{section_strong_syn} and grows exponentially in $\varepsilon^{-1}$. \\
We concentrate on the case of dimension $d=2$ where the process behaves similar to a process on a 
unit sphere which is known to synchronize weakly. 
It remains as an open problem whether one can use similar arguments in higher dimensions as well.
\\
We perform a time change and compare the accelerated process to a process on the unit sphere.
Therefore, we write the accelerated process in polar coordinates. Precisely, we consider
\begin{align*}
	\left( R_t^\varepsilon \cos \phi_t^\varepsilon, 
	R_t^\varepsilon \sin \phi_t^\varepsilon \right) 
	= X_{t/\varepsilon}^\varepsilon.
\end{align*}
Then,
\begin{align}
\label{SDE_radius_square}
	d (R_{t}^\varepsilon)^2 = 
		- \frac{4}{\varepsilon} (R_t^\varepsilon)^2 u'((R_t^\varepsilon)^2) dt 
		+ 2 R_t^\varepsilon \cos \phi_t^\varepsilon d\tilde{W}_t^1 
		+ 2	 R_t^\varepsilon \sin \phi_t^\varepsilon d\tilde{W}_t^2 
		+ 2 dt.
\end{align}
where $u'$ is the first derivative of $u$ and 
\begin{align}
\label{SDE_angle}
	d\phi_t^\varepsilon = \frac{1}{R_t^\varepsilon} \left( - \sin \phi_t^\varepsilon d\tilde{W}_t^1
	+ \cos \phi_t^\varepsilon d\tilde{W}_t^2 \right)
\end{align}
where $\left(\tilde{W}_t^1,\tilde{W}_t^2\right)  = \sqrt{\varepsilon} W_{t/ \varepsilon}$ 
and $(\tilde{W}^1,\tilde{W}^2)$ is a $2$-dimensional Brownian motion.
As $\varepsilon \rightarrow 0$, the drift of $R_t^\varepsilon$ will move the radius close to $1$. Hence, 
we aim to compare $\phi_t^\varepsilon$ to the process
\begin{align}
\label{SDE_limit_cycle}
	dZ_t = - \sin Z_t d\tilde{W}_t^1 + \cos Z_t d\tilde{W}_t^2
\end{align}
on the limit cycle $S = \mathbb{R} / 2 \Pi \mathbb{Z}$.
After we show that $R_t^\varepsilon$ is close to $1$ and $\phi_t^\varepsilon$ is close to $Z_t$, 
we will use that the RDS associated to \eqref{SDE_limit_cycle} is known to synchronize weakly, i.e. every 
point in $S$ converges to a single random point.

\begin{lemma}
	\label{lemma_hitting_of_radius_one}
	Let $0<\alpha<\beta<1$, $T >0$ and $0< r_1 < 1 < r_2 < r_3 < \infty$.
	Then, there exists an $\varepsilon_0 >0$ 
	such that 
	\begin{align*}
		\mathbb{P} \left( r_1 < R_T^\varepsilon < r_2  \right) \geq 1 - \beta
	\end{align*}
	for all $\varepsilon \leq \varepsilon_0$ and any $\mathcal{F}^-$-measurable $X_0^\varepsilon$ satisfying
	\begin{align*}
		\mathbb{P} \left ( R_0^\varepsilon \leq r_3 \right) \geq 1 - \alpha.
	\end{align*}.
\end{lemma}

\begin{proof}
	Choose $k \in \mathbb{N}$ such that $2^{-k+1} \leq \beta - \alpha$ and set
	$t = \min \left\{1/2, T / (2k) \right\}$.
	Using \eqref{SDE_radius_square}, 
	\begin{align*}
		 \int_0^{t} R_t^\varepsilon \cos \phi_t^\varepsilon d\tilde{W}_t^1 
		+ 	\int_0^{t} R_t^\varepsilon \sin \phi_t^\varepsilon d\tilde{W}_t^2 >0
	\end{align*}
	implies that $(R_s^\varepsilon)^2 \geq 2 t$ for some $s \leq t$. 
	Set $r_4 = \sqrt{2t} \leq 1$. Then
	\begin{align*}
		\mathbb{P} \left( R_s^\varepsilon \geq r_4 \textrm{ for some } s \leq t \right) \geq 1/2.
	\end{align*}
	Conditioning on the event $\left\{ R_s^\varepsilon < r_3 \textrm{ for all } s \leq (j-1)t \right\}$
	for $j=2,3,...k$ yields to
	\begin{align}
			\begin{split}
				\mathbb{P} \left( R_s^\varepsilon < r_4 \textrm{ for all } s \leq T /2 \right)
				&\leq \mathbb{P} \left( R_s^\varepsilon < r_4 \textrm{ for all } s \leq kt \right) \\
				&\leq 1/2 \; \mathbb{P} \left( R_s^\varepsilon < r_4 \textrm{ for all } s \leq (k-1)t \right) \\
				&\leq 2^{-k} \leq (\beta - \alpha)/2.
			\end{split}
	\end{align}
	Combining this estimate and the assumption,
	it follows that
	\begin{align*}
		\mathbb{P} \left( r_4 \leq R_s^\varepsilon \leq r_3 \textrm{ for some } s \leq T/2 \right) 
			\geq 1 - (\alpha + \beta)/2.
	\end{align*}
	Let $r_1<r_5<1<r_6<r_2$.
	By Lemma \ref{bound_on_stopping_time}, there exists $C, \varepsilon_1>0$ such that
	\begin{align*}
		\mathbb{P} \left( r_5 < R_s^\varepsilon < r_6 \textrm{ for some } s 
			\leq T/2 +  \varepsilon C \right) 
			\geq 1 - (\alpha + 2\beta)/3.
	\end{align*}
	for all $\varepsilon \leq \varepsilon_1$. Hence, 
	for all $\varepsilon \leq \min \left\{ \varepsilon_1 , T/(2C) \right\}$
	\begin{align*}
		\mathbb{P} \left( r_5 < R_t^\varepsilon < r_6 \textrm{ for some } t \leq T \right) 
		\geq 1 - (\alpha + 2\beta)/3..
	\end{align*}
	Using Proposition \ref{prop_lower_bound}, the statement follows.
\end{proof}

\begin{lemma}
	\label{lemma_convergence_to_limit_cycle}
	Let $0<\alpha < \beta <1$ and $\delta,T>0$. Then, there exists $\varepsilon_0, \eta >0$ such that
	\begin{align*}
		\mathbb{P} \left( \max_{t \leq T} \left| R_t^\varepsilon -1 \right| < \delta \textrm{ and } 
			\max_{t \leq T} \left| \phi_t^\varepsilon - Z_t \right| < \delta \right) \geq 1 - \beta
	\end{align*}
	 for all $\varepsilon\leq \varepsilon_0$ and all $\mathcal{F}^-$-measurable 
	$X_0^\varepsilon$ and $Z_0$ satisfying
	\begin{align*}
		\mathbb{P} \left( \left| R_0^\varepsilon -1 \right| < \eta \textrm{ and } 
			\left| \phi_0^\varepsilon - Z_0 \right| < \eta \right) \geq 1 - \alpha.
	\end{align*}
\end{lemma}

\begin{proof}
	Choose $0< \eta < 0.5 \min \left\{ \delta, 1 \right\} $ such that 
	$\left( 4 \eta^2 + 128 T \eta^2(1 - 2\eta)^{-2} \right) e^{16T} < (\beta - \alpha) \delta^2 $.
	Define
	\begin{align*}
		B_t^\varepsilon := \left\{ \, \max_{s \leq t} \left| R_s^\varepsilon -1 \right| < 2 \eta \right\}
			\cap \left\{ \, \left| \varphi_0^\varepsilon - Z_0 \right| < \eta \right\}.
	\end{align*}
	for all $t\leq T$. Using Proposition \ref{prop_lower_bound} and the assumption, 
	there exists $\varepsilon_0>0$
	such that 
	\begin{align*}
		\mathbb{P} \left( B_T^\varepsilon \right) \geq 
			1 - (\alpha + \beta)/2
	\end{align*}
	for all $\varepsilon \leq \varepsilon_0$.
	We use Doob's inequality and Ito isometry to estimate
	\begin{align*}
		\mathbb{E} \max_{t \leq T} \left| \phi_{t}^\varepsilon - Z_{t} \right|^2 \mathbbm{1}_{B_{T}^\varepsilon}
			&\leq 2 \mathbb{E} \left| \phi_0^\varepsilon - Z_0 \right|^2 \mathbbm{1}_{B_T^\varepsilon}
			+2\mathbb{E} \max_{t\leq T} \bigg(  \int_{0}^{t} 
			\left( \sin Z_s - \frac{1}{R_s^\varepsilon} \sin \phi_s^\varepsilon \right)
			\mathbbm{1}_{B_s^\varepsilon} \, dW_s^1 \\
			& \qquad + \int_{0}^{t} 
			\left( -\cos Z_s + \frac{1}{R_s^\varepsilon} \cos \phi_s^\varepsilon \right)
			\mathbbm{1}_{B_s^\varepsilon} \, dW_s^2 \bigg)^2 \\
		& \leq 2 \eta^2 +4 \mathbb{E} \int_{0}^{T} 
			\left( \left( \sin Z_t - \frac{1}{R_t^\varepsilon} \sin \phi_t^\varepsilon \right)^2
			+ \left( -\cos Z_t + \frac{1}{R_t^\varepsilon} \cos \phi_t^\varepsilon \right)^2 \right)
			\mathbbm{1}_{B_t^\varepsilon} \, dt\\
		& \leq 2 \eta^2 + 16 \,\mathbb{E} \int_{0}^{T} 
			\left( \left| Z_t - \phi_t^\varepsilon \right|^2 + 
			\left| 1 - \frac{1}{R_t^\varepsilon} \right|^2 \right)
			\mathbbm{1}_{B_t^\varepsilon} \, dt \\
		& \leq 2 \eta^2 + 64T \frac{\eta^2}{(1 - 2\eta)^2} + 16 \int_{0}^{T} 
			\mathbb{E} \max_{s \leq t}  \left| Z_s - \phi_s^\varepsilon \right|^2 \mathbbm{1}_{B_t^\varepsilon} \, dt.
	\end{align*}
	Using Gronwall's inequality, it follows that
	\begin{align*}
		\mathbb{E} \max_{t \leq T} \left| \phi_{t}^\varepsilon - Z_{t} \right|^2 \mathbbm{1}_{B_{T}^\varepsilon} 
		\leq \left( 2 \eta^2 + 64T \frac{\eta^2}{(1 - 2\eta)^2} \right) \, e^{16T} < ( \beta - \alpha)
		 \delta^2 /2.
	\end{align*}
	Using Markov inequality, we get
	\begin{align*}
		\mathbb{P} \left( \max_{t \leq T}\left| R_{t}^\varepsilon -1 \right| < \delta \textrm{ and }
			 \max_{t\leq T}\left| \phi_{t}^\varepsilon - Z_{t} \right| < \delta \right)
		&  \geq \mathbb{P}\left( B_{T}^\varepsilon \right) - 
			\mathbb{P} \left( B_{T}^\varepsilon \textrm{ and }
			\max_{t \leq T }\left| \phi_{t}^\varepsilon - Z_{t} \right| \geq \delta \right) \\
		&  \geq 1 - (\alpha + \beta)/2 - \delta^{-2} \, 
			\mathbb{E} \max_{t \leq T} \left| \phi_{t}^\varepsilon - Z_{t} \right|^2 \mathbbm{1}_{B_{T}^\varepsilon}\\
		&  \geq 1-\alpha	 
	\end{align*}
	for all $\varepsilon \leq \varepsilon_0$.
\end{proof}

\subsection{Asymptotic stability of the process on the unit sphere}
The SDE \eqref{SDE_limit_cycle} has a stable point whose Lyapunov exponent is negative, 
see \cite{Baxendale1986}. This random point is the minimal weak point attractor of the RDS
associated to \eqref{SDE_limit_cycle} which we in the following denote by $A^Z$. 
Observe that due to the time change the minimal weak point attractor $A^Z$ of the RDS associated to 
\eqref{SDE_limit_cycle} at time $t$ is $A^Z( \theta_{t/\varepsilon} \omega)$.
When we consider the distance of $A^Z$ to a point in $\mathbb{R}^2$,
we identify with $A^Z$ the point $ \left( \cos A^Z, \sin A^Z \right)$ on the unit sphere.\\
Denote by $Z_t (Z_0)$ the solution of \eqref{SDE_limit_cycle} started in $Z_0$.
We now show the rate of convergence of $Z_t(Z_0)$ to $A^Z$, first for deterministic $Z_0$ and then for
$\mathcal{F}^-$-measurable $Z_0$.

\begin{lemma}
	\label{convergence_lc_to_attractor_determ_start}
	For any $\alpha>0$ and $0 < \mu <1/2$ there exists $C>0$ such that 
	\begin{align*}
		\mathbb{P} \left( \left| Z_t(Z_0) - A^Z (\theta_{t/ \varepsilon} \cdot) \right| \leq C \, e^{- \mu t}
		\textrm{ for all } t \geq 0 \right) \geq 1- \alpha
	\end{align*}
	for all $Z_0 \in [0, 2 \Pi)$. 
\end{lemma}

\begin{proof}
	By \cite{Baxendale1986}, the top Lyapunov exponent of \eqref{SDE_limit_cycle} is $-1/2$. 
	Stable manifold theorem implies
	that for all $0 < \mu <0.5$ there exist a measurable $c (\omega)>0$  and a measurable neighborhood 
	$U(\omega)$ of $A^Z (\omega)$ such that
	\begin{align*}
		\left| Z_t(x) - A^Z (\theta_{t/ \varepsilon} \omega) \right| <  c(\omega) e^{-\mu t}
	\end{align*}
	for all $x \in U(\omega)$ and $t\geq 0$. Hence, for any $\alpha >0$ there exists some $\tilde{c}, \delta >0$
	such that 
	\begin{align*}
		\mathbb{P} \left( \left| Z_t(x) - A^Z (\theta_{t/ \varepsilon} \omega) \right| < \tilde{c} e^{-\mu t} 
		\textrm{ for all } x \in A^Z (\omega)^\delta \textrm{ and } t \geq 0 \right) \geq 1 - \alpha/2.
	\end{align*}
	Since $A^Z (\omega)$ is the attractor of the RDS associated to \eqref{SDE_limit_cycle}, there exists
	a time $T>0$ such that 
	\begin{align*}
		\mathbb{P} \left( \left| Z_T(x) - A^Z (\theta_{T/ \varepsilon} \omega) \right| 
		< \delta \right) \geq 1 - \alpha/2
	\end{align*}
	for all $x \in [0, 2 \Pi)$. Combining these two estimates yields to
	\begin{align*}
		\mathbb{P} \left( \left| Z_t(x) - A^Z (\theta_{t/ \varepsilon} \omega) \right| < \tilde{c} e^{-\mu (t-T)} 
			\textrm{ for all } t \geq T 
		\right) \geq 1 - \alpha
	\end{align*}
	for all $x \in [0, 2 \Pi)$ and $t\geq 0$.
\end{proof}

\begin{proposition}
	\label{convergence_lc_to_attractor_random_start}
	For any $\alpha>0$ and $0 < \mu <0.5$ there exists $C>0$ such that 
	\begin{align*}
		\mathbb{P} \left( \left| Z_t - A^Z (\theta_{t/ \varepsilon} \cdot) \right| \leq C \, e^{- \mu t}
		\textrm{ for all } t \geq 0 \right) \geq 1- \alpha
	\end{align*}
	for all $\mathcal{F^-}$-measurable $Z_0$. 
\end{proposition}

\begin{proof}
	The weak point attractor $A^Z (\omega)$ is an $\mathcal{F^-}$-measurable stable point. Reverting the time,
	one receives an $\mathcal{F^+}$-measurable unstable point $U^Z (\omega)$.
	Hence, $U^Z (\omega)$ and $A^Z (\omega)$ are independent.
	Under the dynamics of \eqref{SDE_limit_cycle} every single deterministic point converges to the attractor.
	However, the unstable point does not converge to the attractor. \\
	If the unstable point is in an interval and the attractor is not, then the time the 
	endpoints of this interval require to approach the attractor is an upper bound for the time any point
	outside the interval requires to approach the attractor.\\
	Let $n \in \mathbb{N}$ such that $ \alpha n \geq 4$. 
	We define
	\begin{align*}
		I_k := \left[ k \frac{2 \Pi}{n}, (k+1) \frac{2 \Pi}{n} \right), \qquad
		P_k= k \frac{2 \Pi}{n} \quad \textrm{ and } \quad P_n=P_0
	\end{align*}
	for $0\leq k < n$. 
	By Lemma \ref{convergence_lc_to_attractor_determ_start} there exists $C>0$ such that
	\begin{align*}
		\mathbb{P} \left( \left| Z_t(P_k) - A^Z (\theta_{t/ \varepsilon} \cdot) \right| > C \, e^{- \mu t}
		\textrm{ for some } t \geq 0 \right) 
		\leq \frac{\alpha}{4n}
	\end{align*}
	for all $0 \leq k \leq n$. If $U^Z(\omega) \in I_k$ and $A(\omega) \not\in I_k$ for some $0 \leq k <n$, then
	\begin{align*}
		\sup_{z \not\in I_k} \left| Z_t(z) - A^Z(\theta_{t/ \varepsilon} \omega ) \right| 
		= \min \left\{ \left| Z_t(P_k) - A^Z(\theta_{t/ \varepsilon} \omega ) \right| , 
		 \left| Z_t(P_{k+1}) - A^Z(\theta_{t/ \varepsilon} \omega ) \right| \right\}
	\end{align*}		
	for all $t \geq 0$. Therefore,
	\begin{align*}
		&\mathbb{P} \left( \, \left| Z_t (Z_0) - A^Z (\theta_{t/ \varepsilon} \cdot) \right| 
			\leq C \, e^{- \mu t}
			\textrm{ for all } t \geq 0 \right) \\
		& \qquad \geq \sum_{k=0}^{n-1} \mathbb{P} 
			\big( U^Z(\cdot) \in I_k, A^Z(\cdot) \not\in  I_k, Z_0 \not\in I_k,
			\left| Z_t (P_k) - A^Z (\theta_{t/ \varepsilon} \cdot) \right| \leq C \, e^{- \mu t} 
			\\ & \qquad \qquad \quad \textrm{ and }
			\left| Z_t (P_{k+1}) - A^Z (\theta_{t/ \varepsilon} \cdot) \right| \leq C \, e^{- \mu t}
			\textrm{ for all } t \geq 0 \big)  \\
		& \qquad \geq \sum_{k=0}^{n-1} \big(  
			\mathbb{P} \left(  U^Z(\cdot) \in I_k \right) 
			\mathbb{P} \left( A^Z(\cdot) \not\in  I_k, Z_0 \not\in I_k \right) 
			\\ & \qquad \qquad \quad
			- \mathbb{P} \left( \left| Z_t (P_k) - A^Z (\theta_{t/ \varepsilon} \cdot) \right| 
			> C \, e^{- \mu t} 
			\textrm{ for some } t \geq 0 \right) 
			\\ & \qquad \qquad \quad
			- \mathbb{P} \left( \left| Z_t (P_{k+1}) - A^Z (\theta_{t/ \varepsilon} \cdot) \right| 
			> C \, e^{- \mu t} 
			\textrm{ for some } t \geq 0 \right) \big) \\
		& \qquad \geq \frac{1}{n} \sum_{k=0}^{n-1} 
			\mathbb{P} \left(  A^Z(\cdot) \not\in  I_k, Z_0 \not\in I_k \right) 
			- \alpha /2 \\
		& \qquad \geq \frac{n-2}{n} - \alpha /2 \geq 1- \alpha
	\end{align*}
	for all $\mathcal{F^-}$-measurable $Z_0$. 
\end{proof}

\subsection{Approaching the point attractor}

Combining the estimates from the previous subsections, we are able to show the rate of convergence of 
$X_{t/\varepsilon}^\varepsilon$ to $A^Z$. 
As a direct consequence, we get
that $ A^Z$ and $A^{X,\varepsilon}_{point}$ are close for small $\varepsilon$ and the upper bound 
for the rate of convergence of $X_{t/\varepsilon}^\varepsilon$ to $A^{X,\varepsilon}_{point}$.
Moreover, we show that $X_t^\varepsilon$ does not approach its attractor on a faster time scale.

\begin{proposition}
\label{prop_upper_bound_weak_syn_Z}
	Let $0 <\alpha <\beta <1$, $r>0$ and $0< \mu <0.5$. 
	Then, there exists $C >0$ such that for all $T_1,T_3>0$
	there exists an $\varepsilon_0 >0$ such that 
	\begin{align*}
		\mathbb{P} \left( \left| X_{t/\varepsilon}^\varepsilon - 
		 A^Z(\theta_{t/ \varepsilon} \cdot ) \right|
		\leq C e^{- \mu (t-T_2) } \textrm{ for all } T_2 \leq t \leq T_2+T_3 \right) \geq 1 - \beta
	\end{align*}
	for all $ 0<\varepsilon \leq \varepsilon_0$, $T_2 \geq T_1$ and all 
	$\mathcal{F}^-$-measurable $X_0^\varepsilon$ satisfying 
	\begin{align*}
		\mathbb{P} \left ( R_0^\varepsilon \leq r \right) \geq 1 - \alpha.
	\end{align*}
\end{proposition}
	
\begin{proof}
	Let $\varepsilon>0$.
	We start the SDE \eqref{SDE_limit_cycle} in $Z_{T_1}^\varepsilon= \phi_{T_1}^\varepsilon$.
	By Proposition \ref{convergence_lc_to_attractor_random_start} there exists $c>0$ such that
	\begin{align*}
		\mathbb{P} \left( \left| Z_t - A^Z (\theta_{t/ \varepsilon} \cdot) \right| \leq c \, 
		e^{- \mu (t-T_1)} \textrm{ for all } t \geq T_1 \right) 
		\geq 1- \alpha /2.
	\end{align*} 
	for all $\varepsilon>0$.
	Using Lemma \ref{lemma_hitting_of_radius_one} and \ref{lemma_convergence_to_limit_cycle}, 
	there exists $\varepsilon_0>0$ such that
	\begin{align*}
		\mathbb{P} \left( \max_{T_1 \leq t \leq T_1+T_3}\left| R_{t}^\varepsilon -1 \right| 
		< e^{- \mu T_3} 
		\textrm{ and }
		\max_{T_1\leq t \leq T_1+T_3}\left| \phi_{t}^\varepsilon - Z_{t}^\varepsilon \right| 
		< e^{- \mu T_3} \right) 
		\geq 1- \alpha /2
	\end{align*}
	for all $\varepsilon \leq \varepsilon_0$.
	Setting $C:= c +2 $ it follows that
	\begin{align*}
		\mathbb{P} \left( \left| X_{t/\varepsilon}^\varepsilon - 
		 A^Z(\theta_{t/ \varepsilon} \cdot)  \right|
		\leq C e^{- \mu (t-T_1) } \textrm{ for all } T_1 \leq t \leq T_1+T_3 \right) \geq 1 - \alpha.
	\end{align*}
	Using the same arguments for the process starting in 
	$X_{(T_2-T_1)/\varepsilon}^\varepsilon$ at time 
	$(T_2-T_1)/\varepsilon$, 
	the statement follows.
\end{proof}

\begin{remark}
	Observe that the statement of Proposition \ref{prop_upper_bound_weak_syn_Z} 
	is not true if one takes the supremum over all
	$t \geq T$ inside the probability term. Precisely, for all $\delta, \varepsilon, T >0$ 
	\begin{align*}
		\mathbb{P} \left( \, 
		\sup_{t \geq T} \, \left| X_{t/\varepsilon}^\varepsilon - 
		 A^Z(\theta_{t/ \varepsilon} \cdot) \right| \leq \delta \right) =0
	\end{align*}
	since the process $X_t^\varepsilon$ leaves a neighborhood of the unit sphere for some 
	$t\geq T/\varepsilon$ almost surely.
\end{remark}

\begin{corollary}
	\label{cor_attractors_close}
	For all $\alpha,\delta,T >0$ there exists an $\varepsilon_0 >0$ such that 
	\begin{align*}
		\mathbb{P} \left(  \inf_{a \in A^{X,\varepsilon}_{point} (\theta_{t} \cdot)} 
		|A^Z(\theta_{t} \cdot)-a |
		\leq \delta \textrm{ for all } 0 \leq t \leq T/\varepsilon \right) \geq 1 - \alpha
	\end{align*}
	for all $ 0<\varepsilon \leq \varepsilon_0$.
\end{corollary}

\begin{proof}
	By the construction of the minimal weak point attractor in \cite[Theorem 23]{Crauel2018}, the
	minimal weak point attractor of \eqref{original_SDE} 
	has a $\mathcal{F}^-$-measurable version. 
	We denote this version also by $A^{X,\varepsilon}_{point}$.
	Using \cite[Theorem III.9]{Castaing1977}, 
	we can select an $\mathcal{F}^-$-measurable $x^\varepsilon (\omega)$ where
	\begin{align*}
		x^\varepsilon (\omega) \in 
		\begin{cases}
			A^{X,\varepsilon}_{point} (\omega) \cap B_2, & 
				\textrm{if } A^{X,\varepsilon}_{point} (\omega) \cap B_2 \not = \emptyset \\
			\mathbb{R}^2, & \textrm{else} .
		\end{cases}
	\end{align*}	
	Since the drift of \eqref{original_SDE} pushes any point outside the unit ball
	towards the unit ball, it holds that 
	\begin{align*}
		\lim_{\varepsilon \rightarrow 0} \mathbb{P} 
		\left( A^{X,\varepsilon}_{point} (\omega) \cap B_2  = \emptyset \right) = 0.
	\end{align*}
	Applying Proposition \ref{prop_upper_bound_weak_syn_Z}, there exist some $\varepsilon_1, s>0$ such that 
	\begin{align*}
		\mathbb{P} \left( \left| X_{t/\varepsilon}^\varepsilon (x^\varepsilon(\cdot)) - 
		 A^Z(\theta_{t/ \varepsilon} \cdot ) \right|
		\leq \delta \textrm{ for all } s \leq t \leq T \right) \geq 1 - \alpha /2
	\end{align*}
	for all $\varepsilon\leq \varepsilon_1$. 
	Since $x^\varepsilon (\omega) \in A^{X,\varepsilon}_{point} (\omega)$ implies that 
	$X_{t}^\varepsilon (x^\varepsilon(\omega)) \in A^{X, \varepsilon}_{point} (\theta_{t} \omega )$,
	there exists $\varepsilon_2 >0$ such that
		\begin{align*}
		\mathbb{P} \left( \inf_{a \in A^{X,\varepsilon}_{point} (\theta_{t} \cdot)} 
		 \left| A^Z(\theta_{t} \cdot) -a \right| 
		\leq \delta \textrm{ for all } s/ \varepsilon \leq t \leq (s+T)/ \varepsilon \right) \geq 1 - \alpha
	\end{align*}
	for all $\varepsilon\leq \varepsilon_2$. Using $\theta_{s/\varepsilon}$-invariance 
	of $\mathbb{P}$, the statement follows.
\end{proof}

\begin{theorem}
	\label{thm_upper_bound_weak_syn_X}
	Let $0 <\alpha <\beta <1$, $r>0$ and $0< \mu <0.5$
	Then, there exists $C >0$ such that for all $T_1,T_3>0$
	there exists an $\varepsilon_0 >0$ such that 
	\begin{align*}
		\mathbb{P} \left( \inf_{a \in A^{X,\varepsilon}_{point} (\theta_{t/ \varepsilon} \cdot)} 
		 \left| X_{t/\varepsilon}^\varepsilon - a \right|
		\leq C e^{- \mu ( t-T_2) } \textrm{ for all } 
		T_2 \leq t \leq T_2 +T_3   \right) 
		\geq 1 - \beta
	\end{align*}
	for all $ 0<\varepsilon \leq \varepsilon_0$, $T_2 \geq T_1$ and all 
	$\mathcal{F}^-$-measurable $X_0^\varepsilon$ satisfying 
	\begin{align*}
		\mathbb{P} \left ( R_0^\varepsilon \leq r \right) \geq 1 - \alpha.
	\end{align*}
\end{theorem}

\begin{proof}
	Apply Proposition \ref{prop_upper_bound_weak_syn_Z} and Corollary \ref{cor_attractors_close}
	and use the triangle inequality.
\end{proof}

\begin{theorem}
	\label{thm_lower_bound_weak_syn}
	For any $\alpha>0$ there exist $\varepsilon_0, \delta, T>0$ such that
		\begin{align*}
			\mathbb{P} \left( \sup_{a \in A^{X,\varepsilon}_{point} (\theta_{t} \cdot)}  
			\left| X_t^\varepsilon  - a \right|
			> \delta \textrm{ for all } 0 \leq t \leq T/ \varepsilon \right) \geq 1 - \alpha
		\end{align*}
		for all $\varepsilon \leq \varepsilon_0$ and all deterministic $X_0^\varepsilon \in \mathbb{R}^2$.
\end{theorem}

\begin{proof}
	Let $ \gamma,T>0$ such that $20 \gamma \leq \alpha \Pi$ and $10T \leq \alpha \gamma^2$.
	By Lemma \ref{lemma_convergence_to_limit_cycle} there exists $ 0<2\delta < \sin \gamma$ and 
	$\varepsilon_0>0$ such that for all $0<\varepsilon \leq \varepsilon_0$
	\begin{align*}
		\mathbb{P} \left( \left| \phi_t^\varepsilon - Z^\varepsilon \right| \leq \gamma \textrm{ for all }
		\sigma^{\delta,\varepsilon} \leq t \leq T \right) \geq 1 - \alpha /5
	\end{align*}	 
	where $\sigma^{\delta,\varepsilon} := \inf \left\{ t \geq 0: \left| R_t^\varepsilon - 1 \right| 
	\leq 2 \delta \right\}$ and $Z^\varepsilon_t$ is the solution to \eqref{SDE_limit_cycle} started in 
	$Z^\varepsilon_{\sigma^{\delta,\varepsilon}} = \phi_{\sigma^{\delta,\varepsilon}}^\varepsilon$. Then,
	\begin{align*}
		&\mathbb{P} \left( \left| X_{t/ \varepsilon} - 
			A^Z(\theta_{t/ \varepsilon} \cdot) \right| 
			> 2 \delta \textrm{ for all } t \leq T \right) \\
		& \qquad \geq \mathbb{P} \left( \left| X_{t/ \varepsilon} 
			-  A^Z(\theta_{t/ \varepsilon} \cdot)\right| 
			> \sin \gamma \textrm{ for all } \sigma^{\delta,\varepsilon} \leq t \leq T \right) \\
		& \qquad \geq \mathbb{P} \left( \left| \phi_t^\varepsilon - 
			A^Z (\theta_{t/ \varepsilon} \cdot) \right| 
			> \gamma \textrm{ for all } \sigma^{\delta,\varepsilon} \leq t \leq T \right) \\
		& \qquad \geq \mathbb{P} \left( \left| \phi_t^\varepsilon -
				Z^\varepsilon_t \right| \leq \gamma \textrm{ and }
			\left| Z^\varepsilon_t- A^Z (\theta_{t/ \varepsilon} \cdot) \right| > 2\gamma 
			\textrm{ for all } \sigma^{\delta,\varepsilon} \leq t \leq T \right) \\
		& \qquad \geq \mathbb{P} \left( \left| Z^\varepsilon_t- A^Z (\theta_{t/ \varepsilon} \cdot) \right| 
			> 2\gamma 
			\textrm{ for all } \sigma^{\delta,\varepsilon} \leq t \leq T \right) -  \alpha /5 .
	\end{align*}	 
	Independence of $Z^\varepsilon_{\sigma^{\delta,\varepsilon}}$ and $A^Z(\cdot)$ implies 
	\begin{align*}
		\mathbb{P} \left( \left| Z^\varepsilon_{\sigma^{\delta,\varepsilon}} - A^Z (\cdot) \right| 
		\leq 4 \gamma \right)
		= \frac{8 \gamma}{2\Pi} \leq  \alpha /5.
	\end{align*}
	Since $T$ was chosen small, it holds that
	\begin{align*}
		&\mathbb{P} \left( \left| Z^\varepsilon_t- A^Z (\theta_{t/ \varepsilon} \cdot) \right| > 2\gamma 
			\textrm{ for all } \sigma^{\delta,\varepsilon} \leq t \leq T \right)\\
		& \qquad \geq \mathbb{P} \left( \left| Z^\varepsilon_{\sigma^{\delta,\varepsilon}} - A^Z (\cdot) \right|
			 > 4 \gamma,
			\left| Z^\varepsilon_t - Z^\varepsilon_{\sigma^{\delta,\varepsilon}} \right| \leq \gamma 
			\textrm{ and }
			 \left| A^Z(\cdot)- A^Z (\theta_{t/ \varepsilon} \cdot) \right| \leq \gamma 
			\textrm{ for all } \sigma^{\delta,\varepsilon} \leq t \leq T \right) \\
		& \qquad  \geq 1 - \alpha /5 
			-\mathbb{P} \left( \max_{\sigma^{\delta,\varepsilon} \leq t \leq T} 
			\left| Z^\varepsilon_t - Z^\varepsilon_{\sigma^{\delta,\varepsilon}} \right| > \gamma  \right)
			- \mathbb{P} \left(  \max_{\sigma^{\delta,\varepsilon} \leq t \leq T} 
			\left| A^Z(\cdot)- A^Z (\theta_{t/ \varepsilon} \cdot) \right| > \gamma  \right) \\
		& \qquad \geq 1 - 3 \alpha /5
	\end{align*}
	Therefore, 
	\begin{align*}
		\mathbb{P} \left( \left| X^\varepsilon_{t/ \varepsilon} - 
			A^Z(\theta_{t/ \varepsilon} \cdot) \right| 
			> 2 \delta \textrm{ for all } t \leq T \right) 
		\geq 1 - 4 \alpha /5.
	\end{align*}	 
	Applying Corollary \ref{cor_attractors_close}, the statement follows.
\end{proof}

For small $\delta>0$ denote by 
\begin{align*}
	\underline{\tau _{0,\delta,x}^\varepsilon} := \inf \left\{ t \geq 0 : 
	\inf_{a \in A^{X,\varepsilon}_{point} (\theta_{t} \cdot)} 
	\left| X_t^\varepsilon (x)  - a \right| \leq \delta \right\}
\end{align*}
and 
\begin{align*}
	\overline{\tau _{0,\delta,x}^\varepsilon} := \inf \left\{ t \geq 0 : 
	\sup_{a \in A^{X,\varepsilon}_{point} (\theta_{t} \cdot)} 
	\left| X_t^\varepsilon (x)  - a \right| \leq \delta \right\}
\end{align*}
the time the process $X_t^\varepsilon$ started in $x \in \mathbb{R}^2$ requires to approach
some point respectively all points of the minimal weak point attractor $A^{X,\varepsilon}_{point}$.
Observe that $ \underline{\tau _{0,\delta,x}^\varepsilon} \leq \overline{\tau _{0,\delta,x}^\varepsilon}$.
If the RDS associated to \eqref{original_SDE} synchronize both quantities coincide.

\begin{corollary}
	\label{cor_time_weak_syn}
	For any $\alpha>0$ there exists some $\delta_0>0$ such that for all $0<\delta \leq \delta_0$ 
	there exist $\varepsilon_0,T_1,T_2>0$ such that
	\begin{align*}
		\mathbb{P} \left( 
		\underline{\tau _{0,\delta,x}^\varepsilon} < T_2 / \varepsilon
		\textrm{ and }
		\overline{\tau _{0,\delta,x}^\varepsilon} > T_1 / \varepsilon \right)
		\geq 1 - \alpha
	\end{align*}	 
	for all $0<\varepsilon \leq \varepsilon_0$ and $ x \in \mathbb{R}^2$.
	In particular, if the RDS associated to \eqref{original_SDE} synchronize weakly, then
	\begin{align*}
		\mathbb{P} \left( 
		T_1 / \varepsilon < \underline{\tau _{0,\delta,x}^\varepsilon} 
		= \overline{\tau _{0,\delta,x}^\varepsilon}  < T_2 / \varepsilon \right)
		\geq 1 - \alpha
	\end{align*}
\end{corollary}

\begin{proof}
	The lower bound follows by theorem \ref{thm_lower_bound_weak_syn} 
	and the upper bound by theorem \ref{thm_upper_bound_weak_syn_X}.
\end{proof}

\begin{remark}
	In contrast to Corollary \ref{cor_time_weak_syn}, if $u$ has more than one local minima the 
	time until a point approach the attractor under the dynamics of \eqref{original_SDE} can 
	increase exponentially in $\varepsilon^{-1}$. 
	For this purpose, observe that one can find a lower bound for the time until
	the paths of the solution started in different minima approach each other using the 
	difference of the potential $U$ in the minima 
	and similar arguments as in section \ref{section_lower_bound}. \\
	Hence, in the case of $u$ having multiple minima, the difference between the time 
	a point and a set requires to approach the attractor is not as significant as in the case
	where $u$ has exactly one minimum. \\
\end{remark}

\section*{Acknowledgement}

The author would like to thank Michael Scheutzow and Anthony Quas for drawing her attention to this problem.

\bibliographystyle{plain}
\bibliography{mybib}

\end{document}